\documentclass[11pt,reqno,oneside]{amsart}

\usepackage{amsmath,amsthm,amssymb}
\usepackage[colorlinks=true,linkcolor=blue, urlcolor=red, citecolor=blue, backref]{hyperref}
\usepackage{graphicx,color}
\usepackage{xcolor}
\usepackage{hyperref}
\usepackage{mathrsfs}
\usepackage[utf8]{inputenc}
\usepackage{wasysym}
\usepackage{dsfont}

\title[Bayes posterior convergence]
{Bayes posterior convergence for loss functions via almost additive Thermodynamic Formalism}

\author[A. O. Lopes]{Artur O. Lopes}
\address{Universidade Federal do Rio Grande do Sul, 91509-900 MAT, Porto Alegre, Brasil}
\email{arturoscar.lopes@gmail.com}

\author[S. R. C. Lopes]{Silvia R. C. Lopes}
\address{Universidade Federal do Rio Grande do Sul, 91509-900 MAT, Porto Alegre, Brasil}
\email{silviarc.lopes@gmail.com}

\author[P. Varandas]{Paulo Varandas $^\dagger$}
\address{Instituto de Matem\'atica e Estat\'istica, Universidade Federal da Bahia, 40170-110, Brazil}
\email{paulo.varandas@ufba.br}
 %Characterization and structure theory of statistical distributions
\date{August, 2021 \\ \qquad $^\dagger$ Email: paulo.varandas@ufba.br (corresponding author)}

\begin{document}

\theoremstyle{plain}
\newtheorem{maintheorem}{Theorem}
\renewcommand{\themaintheorem}{\Alph{maintheorem}}
\newtheorem{maincorollary}{Corollary}
\newtheorem{theorem}{Theorem }[section]
\newtheorem{proposition}[theorem]{Proposition}
\newtheorem{corollary}[theorem]{Corollary}
\newtheorem{claim}{Claim}
\theoremstyle{definition} \theoremstyle{remark}
\newtheorem{remark}[theorem]{Remark}
\newtheorem{example}[theorem]{Example}
\newtheorem{definition}[theorem]{Definition}
\newtheorem{problem}{Problem}
\newtheorem{question}{Question}
\newtheorem{exercise}{Exercise}

\theoremstyle{plain}
\newtheorem{thm}{Theorem}[section]
\newtheorem{cor}[thm]{Corollary}
\newtheorem{property}{Property}[section]
\newtheorem{prop}[thm]{Proposition}
\newtheorem{lemma}[thm]{Lemma}
\theoremstyle{definition}
\newtheorem{exa}[thm]{Example}
\newtheorem{df}[thm]{Definition}
\newtheorem{rmk}[thm]{Remark}

\maketitle

\begin{abstract}
Statistical inference can be seen as information processing involving input
information and output information that updates belief about some unknown parameters.
We consider the Bayesian framework for making inferences about dynamical systems from ergodic observations,
where the Bayesian procedure is based on the Gibbs posterior inference, a decision process generalization of standard Bayesian inference (see~\cite{BHW,JT}) where the likelihood is replaced by the exponential of a loss function.
In the case of direct observation and almost-additive loss functions, we prove an exponential convergence of the a posteriori measures to a limit measure.
Our estimates on the Bayes posterior convergence for direct observation are related 
and extend
%but complementary to 
those in \cite{Nobel1} to a context where loss functions are almost-additive. 
Our approach makes use of non-additive thermodynamic formalism and large deviation properties
 \cite{L4,L2,VZ}.
%instead of joinings. 
%We assume  Lipschitz regularity in some of our results
\end{abstract}

%=====POSSIBLE JOURNALS=====
%https://www.e-publications.org/ims/submission/BA/help/about/
%https://onlinelibrary.wiley.com/journal/15410420?tabActivePane=
%http://www.bernoulli-society.org/publications/bernoulli-journal/bernoulli-journal-editorial-board
%https://imstat.org/journals-and-publications/annals-of-statistics/
% JSP
% PMP https://msp.org/pmp/2020/1-1/
% Nonlinearity
% JMP https://aip.scitation.org/journal/jmp

%\tableofcontents

\section{Introduction and statement of the main results}

\subsection{Bayesian inference}

Statistical inference aims to update beliefs about uncertain parameters as more information becomes available.
The Bayesian inference, one of the most successful methods used in decision theory, builds over Bayes' theorem:
\begin{equation}\label{eq:Bayes}
\text{Prob}(H\mid E) = \frac{\text{Prob} (E\mid H) \cdot \text{Prob}(H) }{\text{Prob}(E)}
	= \frac{\text{Prob} (E\mid H)}{\text{Prob}(E)}  \cdot \text{Prob}(H)
\end{equation}
which expresses the conditional probability of the hypothesis $H$ conditional to the event $E$ with the probability
that the event or evidence $E$ occurs given the hypothesis $H$.
In the previous expression, the \emph{posterior probability} $\text{Prob}(H\mid E)$ is inferred as an outcome of the
\emph{prior probability}  $\text{Prob}(H)$ on the hypothesis, the model evidence $\text{Prob}(E)$ and the likelihood
$\text{Prob} (E\mid H)$. Bayes' theorem has been widely used as an inductive learning model to transform prior and sample information into posterior information and, consequently, in decision theory.
%xxxxxxx
One should not make confusion between $\text{Prob} (E\mid H)$ and $\text{Prob} (H\mid E)$.
Let us provide a simple example. Suppose one is tested for covid-19, and the test turns out to be
positive. If the test is $99 \%$ accurate, the latter means that  $\text{Prob}(\text{Positive test}\mid \text{Covid-}19)=0.99$. %$q(y |c) = A = 0,99$.
However, the most relevant information is $\text{Prob}(\text{Covid-}19 \mid \text{Positive test})$, namely the probability
of having covid-19 once one is tested positive, which is related with the former conditional probability by ~\eqref{eq:Bayes}.
If proportion $\text{Prob}(\text{Covid-}19)$ of infected persons in the total population is $0.001$ it it possible to compute the normalizing term
$\text{Prob}(\text{Positive test})$ and to conclude that $\text{Prob}(\text{Covid-}19\mid \text{Positive test})=0.5$, which provides a different and rather relevant  information (see e.g. \cite{Cat} for all computations in a similar example). 
The conclusion is that both the prior and the data contain important information, and so
neither should be neglected.

The process of drawing conclusions from available information is called inference. However, in many physical phenomena the available information is often insufficient to reach certainty through a logical reasoning.  In these cases, one may use different approaches for doing inductive inference, and the most common methods are those involving probability theory and entropic inference (cf.~\cite{Cat}).
The frequentist interpretation advocates that the probability of a random event is given by the relative number of occurrences of the event in a sufficiently large number of identical and independent trials. 
An alternative approach is given by the Bayesian interpretation %which can be traced back to Bernoulli and Laplace, 
%but have only achieved popularity in the last few decades, 
which became more popular in the recent decades and sustains that a probability reflects the degree of belief of an agent in the truth of a proposition. Citing \cite{Cat}, {``the crucial aspect of Bayesian probability measures is that different agents may have different degrees of belief in the truth of the very same proposition, 
a fact that is described by referring to Bayesian probability measures as being subjective".}
%The term is somewhat misleading because there is a wide variety of views on this matter.

\medskip
In the framework of parametric Bayesian statistics, one is interested in updating beliefs, or the degree of confidence, on the space of parameters $\Theta$, which play the role of the variable $H$ in the expression ~\eqref{eq:Bayes} above. In rough terms, the formula ~\eqref{eq:Bayes} expresses that the belief on a  certain set of parameters is updated from the original belief, after an event $E$, by how likely such event is for all parameterized models. 
This supports the idea that while frequentists say the data are random and the parameters are fixed, Bayesians say the data are fixed and the parameters are random. 
The basic idea in classical Bayesian inference is the updating of a prior belief distribution to a posterior belief distribution when the parameter of interest is connected to observations via the likelihood function.
In \cite{BHW}, Bissiri et al propose a general framework for the Bayesian inference arguing that a valid update of a prior belief distribution to the posterior one can be made for parameters which are connected to observations through a loss function which accumulates information as time passes rather than the likelihood function. In their framework, the classical inference process corresponds to the special case
where the loss function is expressed as the negative log likelihood function.
In this more general framework, the choice of loss function determines the way that the data are analyzed contribute to the mechanism of updating the belief
distribution on the space of parameters, and such choice is often subjective and depends on the kind of feature one desires to highlight from the data.
Moreover, the purpose is that the successive updated belief distributions, called posterior distributions, either converge or concentrate around the unknown 
targeted parameters. We refer the reader to \cite{Abra,Sche1,Denk,Roh,Sche, Suhov} for more information on classical Bayesian inference formalism.

\medskip
The Bayesian inference in the context of observations arising from dynamical systems faces some natural challenges. 
The first one is that the process of taking time series (via Birkhoff theorem) lacks independence: if $T: (Y,\nu)\to (Y,\nu)$ is a measure preserving map 
and $\phi: Y\to\mathbb R$ is an observable then the sequence of random variables $(\phi\circ T^n)_{n\geqslant 1}$ is identically distributed 
but the random variables are not even pairwise independent. The second one concerns the choice of the loss function to make update of beliefs on the
space of parameters. From the Physics and the Dynamical Systems viewpoints it is natural that loss functions should value some of the geometric
or chaotic properties of the dynamical system, identified either in terms of Lyapunov exponents, joint spectral radius of matrix cocycles, entropy
 or estimates on the Charath\'eodory, box-counting or Hausdorff dimension of repellers and attractors, and with applications in wavelets and multifractal analysis, just to mention a few.  These concepts, central in mathematical physics (see e.g.  \cite{Barreira,Barreira2,Benoist,BKL,Bochi,Daube,Feng1,Falc,Galla,Galla1} and references therein) appear naturally as limits of either Birkhoff averages of potentials, sub-additive or almost-additive potentials (or several other versions of non-additivity, to be defined in Subsection~\ref{subsec:almostadditive}). As a first example, if $T$ is a $C^1$-smooth volume preserving and ergodic diffeomorphism on a surface then its largest Lyapunov exponent is 
 %defined Lebesgue almost everywhere 
 by the random product of $SL(2,\mathbb R)$-matrices as
 $$
 \lambda_+(T,\text{Leb})=\lim_{n\to\infty} \frac1n \log \|A(T^{n-1}(y))\dots A(T(y)) A(y)\|
 $$
for Lebesgue almost every $y\in Y$, where $A=DT: Y \to TY$ is the derivative cocycle. In general, the sequence of observables $\Phi=(\varphi_n)_{n\geqslant 1}$ defined by $\varphi_n(y)=\log \|A(T^{n-1}(y))\dots A(T(y)) A(y)\|$ is sub-additive and, in the special case that the linear cocycle has an invariant cone-field, 
this sequence is actually almost-additive (cf. \cite{Feng}). 
 A second example concerns the Shannon-McMillan-Breiman formula for entropy on one-sided subshifts of finite type 
 $\sigma:\Omega\to\Omega$, where the set $\Omega\subset \{1,2,...,q\}^\mathbb{N}$ is $\sigma$-invariant 
 determined by a transition matrix
$M_\Omega \in \mathcal M_{q\times q}(\{0,1\})$
and
\begin{equation}\label{eq:SMB}
h_{\mu}(\sigma)=\lim_{n\to\infty} -\frac1n \log \mu(C_n(x)), \quad\text{ for } \mu\text{-a.e. $x$}
 \end{equation}
where $C_n(x)\subset \Omega$ denotes the $n$-cylinder set %in the shift space 
containing the sequence $x=(x_1, x_2, x_3, \dots)$.
The sequence of observables $\Phi=(\varphi_n)_{n\geqslant 1}$ defined by $\varphi_n(y)=-\log \mu(C_n(x))$, which is non-additive in general, is 
additive and almost-additive in the relevant classes of Bernoulli and Gibbs measures, respectively (see Lemma~\ref{le:ln-mu}). 
Finally, it is worth to mention that sub-additive and almost-additive sequences appear naturally also in applications to several other areas of know\-ledge and appear for instance in the study of factorial languages by Thue, Morse and Hedlund in the beginning of the twentieth century (see \cite{Shur} and references therein).

\medskip

In this article, inspired by the relevant physical quantities arising from non-additive sequences of potentials, we will establish a bridge between non-additive thermodynamic formalism of dynamical systems and  Gibbs posterior inference  
%posterior principle for Bayesian inference 
in statistics (to be defined in Subsection~\ref{subsec:posterior} below), two areas of research in connection with statistical physics. 
We refer the interested reader to the introduction of \cite{Nobel1} for a careful and wonderful exposition on the link between
Bayesian inference and thermodynamic formalism, and a list of cornerstone contributions.
We will mostly be interested in the parametric formulation of Bayesian inference, as described below.
Let $\sigma: \Omega\to\Omega$ be a subshift of finite type. 
This will serve as the underlying dynamical system, with respect to which 
samplings are obtained along its finite orbits $\{y, \sigma(y), \dots, \sigma^{n-1}(y)\}$, $y\in \Omega$. 
%The models and as the defining object with respect to which models
%
%%This will serve simultaneously as the underlying dynamical system, with respect to which 
%samplings are obtained along its finite orbits $\{x, \sigma(x), \dots, \sigma^{n-1}(x)\}$, $x\in \Omega$, and as the defining object with respect to which models
We take a family of Gibbs probability measures $\{\mu_\theta\}_{\theta\in \Theta}$ as the models in the inference procedure %and prime example 
for their relevance and ubiquity in the thermodynamic formalism of dynamical systems, 
%illustrate beautifully the previous connection, 
and are of crucial importance in several other fields as in the study of the randomness of time-series, decision theory, quantum information and information gain, just to mention a few 
%when comparing statistical models of inference 
(cf. \cite{Altaner,Cat,FLL,Ji,LR}). 
In our context, Gibbs measures appear as fixed points of the dual of certain transfer operators. Let us be more precise.
For any  Lipschitz continuous potential $A: \Omega\to  \mathbb{R}$, the Ruelle-Perron-Frobenius transfer operator associated to $A$ is defined by
$$ \mathcal{L}_A (\varphi) (x) = \sum_{\sigma(y) =x} e^{A(y)} \varphi(y).$$
The potential $A$ is called normalized if $ \mathcal{L}_A(1)=1$, and in this case, it is natural to write $A =\log J$, and we call $J$ the Lipschitz  continuous Jacobian. 
A Gibbs measure $\mu$ is any $\sigma$-invariant probability measure obtained as a fixed point of the dual operator $ \mathcal{L}_{\log J}^*$ acting on the space of probability measures on $\Omega$, for some Lipschitz continuous and normalized Jacobian $J$. In this way, it is natural to parametrize Gibbs probabilities by the space of normalized Lipschitz continuous Jacobians $J$, hence this space can be observed as an infinite dimensional Riemannian analytic manifold \cite{GKLM,LR0,LR}.
Invariant Gibbs measures are equilibrium states, namely they satisfy a variational relations (cf.  Subsection~\ref{sec:main} for more details).
Given a prior probability measure $\Pi_0$ on the space $\Theta$ of parameters and a the sampling according to a Gibbs measure $\mu_{\theta_0}$, the posterior probability (i.e. updated belief distribution) is determined using the loss functions  
$
\ell_n : \Theta    \times \Omega \times    \Omega \to \mathbb{R},
$
where $\ell_n(\theta,x,y)$ encodes the information on the parameter $\theta$ accumulated along the sampling $\{y, \sigma(y), \dots, \sigma^{n-1}(y)\}$
and influenced by the measurements along the orbit $\{x, \sigma(x), \dots, \sigma^{n-1}(x)\}$. The Shannon-McMillan-Breiman formula ~\eqref{eq:SMB} 
suggests the use of loss functions to collect the information of the measure on cylinder sets in $\Omega$  (cf. expressions \eqref{tri0}, ~\eqref{tri} and ~\eqref{Chalk} below).
The relative entropy, also called Kullback-Leibler divergence and defined by \eqref{chazo},
makes the comparison between the measurements of cylinders according to two different Gibbs measures. 
This notion is of paramount importance in Physics %(see \cite{Altaner,Cat,Sch,Sa}) 
and will be used to interconnect log likelihood inference 
with the direct observation analysis of Gibbs probability measures. 
Our main results guarantee that posterior consistency for certain classes of loss functions determined by almost-additive sequences of potentials: the
posterior distributions asymptotically concentrate around the unknown targeted parameter $\theta_0$, often with 
exponential speed (we refer the reader to Theorems~\ref{thm:main}, ~\ref{thm:main2} and ~\ref{thm:main3} for the precise statements).
The main ingredient to obtain quantitative estimates on the convergence for the parameter $\theta_0$ is the use of large deviations for non-additive sequences
of potentials \cite{VZ}. 

\medskip
Our results are strongly inspired, and should be compared, with those by McGoff, Mukherjee and Nobel  \cite{Nobel1}, where the authors 
established posterior consistency of (hidden) Gibbs processes on mixing subshifts of finite type using properties of Gibbs measures. 
For that purpose, they consider a more general framework, where the dynamical system $T: Y \to Y$ on a Polish space does not necessarily
coincide with the subshift of finite type $\sigma: \Omega\to\Omega$. In particular, the sampling is determined by a $T$-invariant and ergodic probability 
measure $\nu$, that could be unrelated to the Gibbs measures $\{\mu_\theta\}_{\theta\in \Theta}$ for the shift.
If the loss functions are additive (i.e. $\ell_n=\sum_{j=0}^{n-1} \ell(\theta, \sigma^j(x), T^j(y))$ for some function $\ell: \Theta\times \Omega\times Y \to \mathbb R$
satisfying a mild regularity condition then the main results in \cite{Nobel1} ensure that it is possible to 
formulate  the problem as a limiting variational problem and to identity the parameters, obtained as minimizing parameters for a lower semicontinuous 
function $V:\Theta \to \mathbb{R}$, 
% is introduced in Theorem 1 in page 12.  
for which the posterior consistency holds: if $\Theta_{\min} =\text{argmin}_{\theta \in \Theta} V(\theta)$ % In page 13 it is defined
then the posterior distributions $\Pi_n(\cdot\mid y)$, defined by ~\eqref{triy}, satisfy
$$\lim_{n \to \infty} \Pi_n(\Theta \setminus U\,|y)
=0$$ for each open neighborhood $U$ of $\Theta_{\min}$ and for $\nu$-almost every $y \in {Y}$
(cf. \cite[Theorem~2]{Nobel1}).
The proof of this result requires the use of joinings (or couplings) of the model system and the observed system, and results on fibered entropy.
Our framework corresponds to the special case of direct observation, that the dynamical system $T$ coincides with the subshift of finite type $\sigma$
and the target parameter is a single $\theta_0\in \Theta$, 
with a subtler difference that our assumptions ensure that $\mu_\theta\neq\mu_{\tilde\theta}$ for every distinct $\theta,\tilde\theta\in \Theta$.
Our results complement the ones in \cite{Nobel1} in the sense that the information can 
be collected by more general loss functions $\ell_n$.
Furthermore, the more direct use of large deviation techniques allows to prove an exponential 
speed of convergence in the posterior consistency (cf. Theorem~\ref{thm:main}),
which were not known even in the context of direct observation 
 (cf. \cite[Theorem 2 and Remark 8]{Nobel1}). Summarizing, the three main 
novelties are the extension to non-additive loss functions, the exponential rate of convergence and the proof which is not based on joinings and fiber entropy.
It is also worth noticing that, more recently, Su and Mukherjee \cite{SM} also used a large deviations approach for posterior consistency, using Varadhan's large deviation principle 
for stochastic processes. A different point of view of the Bayesian {\it a priori\,} and {\it a posteriori} formalism will appear in \cite{ELLM}  where results on  thermodynamic formalism for plans are used (see \cite{LM2,LMMS1}). In \cite{Ji} the author considered log-likelihood estimators  in classical thermodynamic formalism and the inference concerns H\"older potentials and not probabilities.

\medskip

To finalize, one should mention that there is an increasing interest to explore the strong connection between Statistical Inference and Physics in general. 
There are several such connections in this regard, including a Bayesian approach to the dynamics of the classical ideal gas  \cite[Section 31.3]{von},
prior sensitivity in the Bayesian model selection context to some galaxy data sets \cite{CaPe}.
In the monograph \cite{Cat}, the author clarifies the conceptual foundations of Physics by deriving the fundamental laws of statistical 
mechanics and of quantum mechanics as examples of inductive inference, while he also advocates that, in view of the fact that models may need to change as time evolves, it may be the case that all areas of Physics may be modeled using inductive inference.

%\subsection{Gibbs posterior distributions}\label{subsec:posterior}
\subsection{Gibbs posterior inference}\label{subsec:posterior}

%\color{red}
%() The basic ingredients of this approach in the context of dynamical systems are the following. There is an observed dynamical system (the analogue of the data in the context of Statistics) and a model family which is here a family of Gibbs measures on a shift of finite type indexed by a parameter ? (belonging to a compact space ?). There is also a loss function which determines a Gibbs posterior distribution which accumulated information as time passes. The goal is to show that this ?posterior? asymptotically concentrates around the unknown targeted parameter ?0. This is done by using classical large deviation results.
%
%
%\color{black}

According to the Gibbs posterior paradigm \cite{BHW,JT}, the beliefs should be updated according to the Gibbs posterior distribution. Let us recall the formulation of this posterior measure following \cite{Nobel1}.

\subsubsection*{Observed system}
Assume that $Y$ is a complete and separable metric space and that $T:Y \to Y$ is a Borel measurable map endowed with a
$T$-invariant, ergodic probability measure $\nu$. This dynamical system represents the observed system and will be used to update
information for the model. This is the analogue of the data in the context of Statistics. The updated belief, given by the \emph{a posteriori} measure, is obtained by feeding data obtained from the observed system on a model by means of a loss function.

\subsubsection*{Model families}
Consider a transitive subshift of finite type $\sigma \colon \Omega \to \Omega$ where $\sigma$ denotes the right-shift map,
acting on a compact invariant set $\Omega\subset \{1,2,...,q\}^\mathbb{N}$ determined by a transition matrix
$M_\Omega \in \mathcal M_{q\times q}(\{0,1\})$. The map $\sigma$ presents different statistical behaviors (e.g. measured in terms of different convergences 
for Ces\`aro averages of continuous observables) according to any of
its equilibrium states associated to Lipschitz continuous observables, each of which satisfies a Gibbs
property (see e.g. Remark 1 in \cite[Section 2]{PP} or \cite{L3}).

Consider a compact metric space $\Theta$ and a family of $\sigma$-invariant probability measures
$$
\mathscr G= \big\{\mu_\theta \colon \theta \in \Theta\big\}
$$
so that:
(i) for every $\theta\in \Theta$ the probability measure  $\mu_\theta$ is a Gibbs measure associated to a Lipschitz continuous potential
$f_\theta : \Omega \to \mathbb R$, that is, there exists $K_\theta>1$ and $P_\theta\in \mathbb R$ so that
\begin{equation}\label{eq:Gibbs}
\frac1{K_\theta} \leqslant \frac{\mu_\theta(C_n(x))}{e^{-n P_\theta} + S_nf_\theta(x)} \leqslant K_\theta, \qquad \forall n\geqslant 1,
\end{equation}
where $S_n f_\theta=\sum_{j=0}^{n-1} f_\theta\circ \sigma^j$ and $C_n(x)\subset \Omega$ denotes the $n$-cylinder set in the shift space $\Omega$ containing the sequence $x=(x_1, x_2, x_3, \dots)$; and (ii) the family $\Theta\ni \theta \mapsto f_\theta$ is continuous
(in the Lipschitz norm). We assume Gibbs measures to be normalized, hence probability measures.
It is well known that the previous conditions ensure the continuity of the pressure function $\Theta\ni \theta \mapsto P_\theta$ and of the map $\Theta\ni \theta\mapsto \mu_\theta$ (in the weak$^*$ topology) \cite{PP}.
In particular, one can take a uniform constant $K>0$ in ~\eqref{eq:Gibbs}.
The problem to be considered here involves a formulation and analysis of an iterative procedure
(based on sampling and updated information) on the family $\mathscr G$ of models.

\subsubsection*{Loss functions and Gibbs posterior distributions}
Consider the product space $\Theta \times \Omega$ endowed with the metric $d$ defined as
$ d(\,(\theta, x),(\theta ' , x ')\,)= \max \{ d_\Theta(\theta , \theta ') ,
d_\Omega   (x, x')\,  \}.$
A fully supported probability measure  $\Pi_0$ on $\Theta$ describes the {\it a priori} uncertainty on the Gibbs measure.

\smallskip
Given such an {\it a priori} probability measure  $\Pi_0$ on the space of parameters $\Theta$ and a sample of size $n$ (determined by the observed system $T$) we will get the {\it a posteriori} probability measure  $\Pi_n$ on the space of parameters
$\Theta$, taking into account the updated information from the data.
More precisely, given $\Pi_0$ and a family $(\mu_\theta)_{\theta \in \Theta}$, consider the probability measure 
$P_0$ on the product space $\Theta \times \Omega$ given by
$$
P_0 (E)= \int \int \mathbf 1_E (\theta,x) \,d\mu_\theta (x) \,d\Pi_0 (\theta)
$$
for all Borel sets $E \subset \Theta \times \Omega$. In other words, $P_0$ has the \emph{a priori} measure $\Pi_0$ as marginal on $\Theta$ and admits a disintegration on the partition by vertical fibers where the fibered measures  are exactly
the Gibbs measures $(\mu_\theta)_{\theta\in \Theta}$. There is no action of the dynamics $T$ on this product space.
Indeed, the \emph{a posteriori} measures are defined %considered by means of interactions 
using loss functions.
For each $n \in \mathbb{N}$  consider a continuous  \emph{loss function} $\ell_n$ of the form
$$
\ell_n : \Theta    \times \Omega \times    Y \to \mathbb{R},
$$
%
%\marginpar{\tiny \cite{Nobel1} needed further regularity on loss functions}
%
consider the probability measure  $P_n$ on $\Theta\times \Omega$ given by
\begin{equation} \label{tri0}
P_n (E\mid y)= \int \int \mathbf 1_E (\theta,x) e^{ - \, \ell_n (\theta, x,y)}\,d\mu_\theta (x) \,d\Pi_0 (\theta)
\end{equation}
for all Borel sets $E \subset \Theta \times \Omega$,
and set
\begin{equation} \label{tri} Z_n(y)=\int_\Theta \int_\Omega\,  e^{ - \, \ell_n (\theta, x,y)}\,d \mu_\theta (x)\, d \Pi_0 (\theta), \end{equation}
where $x=(x_1,x_2,...,x_n,\dots)\in \Omega$  and $y\in Y$. In the special case that $Y=\Omega$, that
$-\ell_n: \Theta \times \Omega \times \Omega \to\mathbb R$ coincides with a $n$-Birkhoff sum of a fixed observable $\psi$ with respect to $T$ and $\Pi_0$ is a Dirac measure,
the expression ~\eqref{tri} resembles the partition function
in statistical mechanics whose exponential asymptotic growth coincides with the topological pressure of $T$ with respect to  $\psi$.

\smallskip
Given $y \in Y$ and $n\geqslant 1$, the {\it a posteriori} Borel probability measure
$\Pi_n (\cdot \,|\, y)$ on the parameter space $\Theta$ (at time $n$ and determined by the sample of $y$)
is defined by
\begin{equation} \label{triy}
\,\Pi_n (B \,|\, y) =\frac1{Z_n(y)}  \int_B \int_\Omega  {e^{ - \, \ell_n (\theta, x,y)}d \mu_\theta (x)}d \Pi_0 (\theta)\,,
%\,\Pi_n (.\,|\, y) =\int  \frac{e^{ - \, \ell_n (\theta, x,y_1,y_2,..,y_n)}d \mu_\theta (x)}{Z_n(y)}d \Pi_0 (\theta)\,.
\end{equation}
%For fixed $y= (y_1,y_2,...,y_n,...)$ and $n\in \mathbb{N}$ we call the probability $A \to \Pi_n (A\,|\, y)$, $A$ in the Borel
%sigma-algebra of $\Theta$, the {\it a posteriori} probability at time $n$ for the sample $y$.
for every measurable $B\subset \Theta$ and appears as marginals of the probability measures $P_n(\cdot\mid y)$ given above.

\smallskip
The general question is to describe the set of probability measures $\,\Pi_n (.\,|\, y)$ on the parameter space
$\Theta$, namely if their marginals converge and to formulate the locus of convergence in terms of some variational principle or as points of maximization for a certain function (see e.g. \cite[Theorem~2]{Nobel1} for a context where the
loss functions are chosen such that the support of such measures on the minimization locus of a certain rate function).

\smallskip
The main problem we are interested in is to understand
whenever a sampling process according to a fixed probability
measure can help to identify it from a recursive process involving
Bayesian inference. %Let us illustrate this with a particular context.
Assume that $Y=\Omega$, that $T=\sigma$ is the shift and that
one is interested in a specific probability measure  $\mu_{\theta_0}\in \mathscr G$,
where $\theta_0 \in \Theta$. If $\nu=\mu_{\theta_0}$ then the sampling
$\{y, T(y), T^2(y), \dots T^{n-1}(y)\}$ is distributed according to this probability measure.
From the Birkhoff time series is it possible to successively update the initial a priori probability measure  $\Pi_0$ in order to get a sequence of probability measures $\Pi_n(\cdot\mid y)$ on $\Theta$ (the a posteriori probability measure  at time $n$) as described.
We ask the following:
\smallskip
\begin{itemize}
\item[$\circ$] Does the limit $\lim_{n \to \infty} \Pi_n$ exist?  \smallskip
\item[$\circ$] If the previous question has an affirmative answer: \smallskip
	\begin{itemize}
	\item[$\circ$] is it the Dirac measure $\delta_{\theta_0}$ on $\theta_0\in \Theta$?
	\item[$\circ$] is it possible to estimate  the speed of convergence to the limiting measure?
	\end{itemize}
\end{itemize}

%Our main interest here is to formulate and analyze the iterative procedure (based on the sampling)
%and to provide answers 
In this paper we answer the previous questions for loss functions that are not necessarily arising from
Birkhoff averaging but that keep some almost additive property. For that reason our approach will make use
of results from non-additive thermodynamic formalism, hence it differs from the one considered in \cite{Nobel1}.
We refer the reader to \cite{Cha1} for a related work which does not involve Bayesian statistics.

This paper is organized as follows. In the rest of this first section we formulate the precise setting we are interested in and state the main results. In Section~\ref{sec:examples} we present several examples and applications of our results. Section~\ref{se:prelim} is devoted to some preliminaries on relative entropy, large deviations and non-additive thermodynamic formalism. Finally, the proofs of the main results are given in Section~\ref{sec:proofs}.

\subsection{Setting and Main results}\label{sec:main}

Let $\sigma :\Omega \to \Omega$ be a subshift of finite type endowed with the metric $d_\Omega(x,y)=2^{-n(x,y)}$, where
$n(x,y)=\inf\{n\geqslant 1\colon x_n\neq y_n\}$, and denote by $\mathcal M_\sigma(\Omega)$ the space of
$\sigma$-invariant probability measures.
The space $\mathcal M_\sigma(\Omega)$ is metrizable and we consider the usual topology on it (compatible with weak$^*$ convergence). Let $D_\Omega$ be a metric on $\mathcal M_\sigma(\Omega)$ compatible with the weak$^*$ topology.
The set $\mathcal{G}\subset \mathcal M_\sigma(\Omega)$ of Gibbs measures for
Lipschitz continuous potentials is dense in $\mathcal M_\sigma(\Omega)$ (see for instance \cite{L2}). Given a Lipschitz continuous potential $A:\Omega \to \mathbb{R}$ we denote by $\mu_A$ the associated Gibbs measure.  We say that the Lipschitz continuous potential $A:\Omega \to \mathbb{R}$ is \emph{normalized} if $\mathcal{L}_A (1)=1$, where
$$
\mathcal{L}_A \colon \text{Lip}(\Omega,\mathbb{R}) \to \text{Lip}(\Omega,\mathbb{R})
	\quad\text{given by}\quad
	\mathcal{L}_A g(x)=\sum_{\sigma(y)=x} \, e^{A(y)} \, g(y)
$$
is the usual Ruelle-Perron-Frobenius transfer operator (cf. \cite[Chapter~2]{PP}).
We will always assume that potentials are normalized and write $J=e^A>0$ (or alternatively $A=\log J$) as the Jacobian of the associated probability measure 
$\mu_A=\mu_{\log J}.$ That is, $\mathcal{L}_{\log J}^* (\mu_{\log J})= \mu_{\log J}$ and, equivalently,
$\mu_{\log J}(\sigma(E)) = \int_E J \, d\mu_{\log J}$ for every measurable set $E\subset \Omega$ so that $\sigma\mid_E$ is injective.  We consider the Lipschitz norm $|\,.\,|=\|\cdot\|_\infty+|\cdot|_{Lip}$ on the space of Lipschitz continuous potentials $A$, where
$
|A|_{Lip} = \sup_{x\neq y} \frac{|A(x)-A(y)|}{d_\Omega(x,y)}.
$
 Moreover, it is a classical result in thermodynamic formalism (see e.g. \cite{PP}) that
the following variational principle holds
\begin{equation} \label{mw}
%\sup_{\mu \in \mathcal M_\sigma(\Omega)} \{\, h(\mu) \,+\, \int \log J_{\theta_0} \, d \mu \,\}= 0.
\sup_{\mu \in \mathcal M_\sigma(\Omega)} \Big\{\, h(\mu) \,+\, \int \log J \, d \mu \,\Big\}= 0
\end{equation}
for any Lipschitz and normalized potential $\log J$.
A particularly relevant context is given by the space of stationary Markov probability measures on shift spaces
(cf. Example~\ref{exlo}). One should emphasize that, replacing the metric on $\Omega$, it is possible to deal instead 
with the space of Lipschitz continuous potentials (cf. \cite[Chapter~1]{PP}).  

\medskip
In the direct observation context, the sampling on the Bayesian inference is determined by
$T=\sigma$ and a fixed $T$-invariant Gibbs measure $\nu $ on $\Omega$ associated
to a normalized potential $\log J$. The sampling will describe the interaction (expressed in terms of the loss functions)
over certain families of potentials (and Gibbs measures) which are parameterized  on a compact set, where the sampling will occur.
More precisely, consider the set of parameters $\Theta\subset \mathbb{R}^k$ of the form
$$
\Theta = [a_1,b_1] \times [a_2,b_2] \times ...\times [a_k,b_k],
$$
endowed with the metric $d_\Theta$ given by $d_\Theta(\theta_1,\theta_2) = \,\|\theta_1 - \theta_2\|$, $\forall \theta_1,\theta_2\in \Theta$, and denote by $f:\Theta \to \mathcal{G}\subset \mathcal M_\sigma(\Omega)$ a continuous function of potentials parameterized over
$\Theta$ such that:
\begin{enumerate}
\item $f$ is an homeomorphism over its image;
\item for each $\theta$ the potential $f(\theta)$ is normalized (we use the notation $f(\theta) = \log J_\theta$).
\end{enumerate}
The assumptions guarantee that for each $\theta \in \Theta$ there exists a unique invariant Gibbs measure
$\mu_\theta$ with respect to the associated normalized potential $f(\theta)$, and that these vary continuously
in the weak$^*$ topology. Moreover, as the parameter space $\Theta$ is compact and $f:\Theta \to \mathcal{G}$ is a continuous function (expressed
in the form $f(\theta)= \log J_\theta$, where $f$ is a continuous function on $\theta \in \Theta$ and $J_\theta>0$), we deduce that
the quotient
\begin{equation} \label{372} \frac{ J_{\theta_1} (x) }{J_{\theta_2} (x)   }>0\end{equation}
is uniformly bounded for every $x \in \Omega$ and all $\theta_1,\theta_2 \in \Theta.$

\begin{remark}
At this moment we are not requiring the probability measure  $\nu$ of the observed system $Y=\Omega$ to belong to the family
of probability measures $(\mu_\theta)_{\theta\in \Theta}$. We refer the reader to Example~\ref{ex2} for an application in the  special case that $\nu=\mu_{\theta_0}$, for some $\theta_0\in \Theta$.
\end{remark}

%\medskip
%\subsection*{\emph{A priori} Bayes probability}
The statistics is described by an \emph{a priori} Bayes probability measure  $\Pi_0$ on the space of parameters $\Theta$ satisfying \emph{Hypothesis A}:
\begin{align*}\tag{A}
\Pi_0 (d z_1, d z_2,...,d z_k)=\Pi_0 (d \theta) \quad
	& \text{is a fixed continuous strictly positive  density } \\
	& \text{fully supported on the compact  set $\Theta$.}
\end{align*}
In many examples the \emph{a priori} measure appears as the Lebesgue or an equidistributed measure on the parameter space. We refer the reader to Section~\ref{sec:examples} for examples. 
% As a possible example under the assumptions  of example \ref{exlo} one can consider $\Pi_0$ as the Lebesgue probability on $\Theta =(0,1) \times (0,1)$.

\smallskip

The previous full support assumption not only expresses the uncertainty on the choice of the parameters, as it ensures that
all parameters in $\Theta$ will be taken into account in the inference independently of the initial belief (distribution of $\Pi_0$).
%Observe that the case $\Pi_0=\delta_{\theta_1}$ represents full belief that $\theta_1$ is the correct parameter,
%hence the inference process would not be necessary.
In this case of direct observations of Gibbs measures,
let $\theta_0\in\Theta$ be fixed. The probability measure  $\mu_{\theta_0}$ will play the role of the measure $\nu$ (on the observed system $Y$) considered abstractly on the previous subsection.
We will consider the {\em loss functions}
$\ell_n : \Theta \times \Omega \times \Omega \to \mathbb R$, $n\geqslant 1$, given by
%
%\marginpar{\tiny this loss function depends on $\theta_0$ BUT independs on $\theta$; is this desired?}
%
\begin{equation} \label{Chalk}
\ell_n(\theta, x,y)=
\begin{cases}
\begin{array}{ll}
%- \log \Big(\frac{\mathbf 1_{C_n(y)}}{\mu_{\theta_0}  \,(\,C_n(y) \,)}\, \Big)
 \log \Big( \mu_{\theta_0}  \,( C_n(y) )\, \Big)
	&, \text{if } x\in C_n(y) \\
	+\infty &, \text{if } x\not \in C_n(y).
\end{array}
\end{cases}
%\ell_n(\theta, x,z)=
%\begin{cases}
%\begin{array}{ll}
%\log \Big(\frac{\mu_{\theta_0}  \,(\,C_n(z) \,)}{\mathbf 1_{C_n(z)} (x )}\, \Big)
%	&, \text{if } x\in C_n(z) \\
%	+\infty &, \text{if } x\not \in C_n(z),
%\end{array}
%\end{cases}
\end{equation}
%\smallskip
%Recall the {\em loss function} for the case of direct observations was defined in ~\eqref{Chalk}.
%Given $\theta_0$ fixed as above, for each $n \in \mathbb{N}$, and any points $x \in \Omega$ and  $z=(z_1,z_2,..,z_n,...)\in \{1,2,..,d\}^\mathbb{N}=Y$ consider
%\begin{equation} \label{Chalk}
%\ell_n(\theta, x,z)=
%\begin{cases}
%\begin{array}{ll}
%\log \Big(\frac{\mu_{\theta_0}  \,(\,C_n(z) \,)}{\mathbf 1_{C_n(z)} (x )}\, \Big)
%	&, \text{if } x\in C_n(z) \\
%	+\infty &, \text{if } x\not \in C_n(z),
%\end{array}
%\end{cases}
%\end{equation}
%where $\mathbf 1_{C_n(z)}$ is the indicator function of the $n$-cylinder set ${C_n(z)}$.
%%when $\mathbf 1_{C_n(z)} (x )=1$. When, $\mathbf 1_{C_n(z)} (x )=0$ we set  $\ell_n(\theta, x,z)=\infty$.
%%

%Taking, alternatively,  $\ell_n(\theta, x,y)=\log \Big( \frac{\mu_{\theta_0}  \,( C_n(y) )}{ \mu_{\theta}  \,( C_n(y) )}\, \Big)$, our reasoning could be adapted to this situation and we would obtain the same results we will get.

If one denotes by $\mathbf 1_{C_n(y)}$ the indicator function of the $n$-cylinder set centered at $y$ and defined by ${C_n(y)}=\{(x_j)_{j\geqslant 1} \colon x_j=y_j, \forall 1\leqslant j \leqslant n\}$, such choice of loss functions ensures that
\begin{align*}
Z_n(y) & =  \int_\Theta \int e^{ - \, \ell_n (\theta, x,y)} \,d \mu_\theta (x)\, d \Pi_0 (\theta)
	 =  \int_\Theta \int_{C_n(y)} e^{ - \, \ell_n (\theta, x,y)} \,d \mu_\theta (x)\, d \Pi_0 (\theta)  \\
	& = \int_\Theta \int     \frac{ \,\,\mathbf 1_{C_n(y)} (x )\,}{\mu_{\theta_0} \,(\,C_n(y)\,) }   \,d \mu_\theta (x)\, d \Pi_0 (\theta)
%	& = \sum_{\overline{ a_1,a_2,...,a_n }  }\int   \int_\Theta  \frac{ \,\,\mathbf 1_{C_n(y)} (x )\,}{\mu_{\theta_0} \,(\,C_n(y)\,) }    \,d \mu_\theta (x)\, d \Pi_0 (\theta)  \\
	 = \int_\Theta \,  \frac{\mu_{\theta}  \,(\,C_n(y) \,)}{\mu_{\theta_0} \,(\,C_n(y)\,) }     \,d \Pi_0 (\theta)
\end{align*}
for each $y\in Y$.
%\color{gray}
Therefore, using equalities \eqref{Chacha} and \eqref{chazo} %$\nu=\mu_{\theta_0}$,
(see Subsection~\ref{subsecrelent} below),
Jensen inequality and the monotone convergence theorem, %limit under the sign of the integral, one
one obtains that
%\color{gray}
%\begin{align}
%- \limsup_{n \to \infty} \frac{1}{n} \log Z_n(y) \,\, &  = -  \limsup_{n \to \infty} \frac{1}{n} \log \int_\Theta  \frac{\mu_{\theta}
%	\,(\,C_n(y) \,)}{\mu_{\theta_0} \,(\,C_n(y)\,) }   d \Pi_0 (\theta) \nonumber \\
%	& \leqslant    -  \limsup_{n \to \infty} \frac{1}{n} \int_\Theta \log  \frac{\mu_{\theta}  (C_n(y) )}{\mu_{\theta_0} (C_n(y)) }
%	%d \mu_{\theta_0}  (y)
%	\, d \Pi_0 (\theta) \nonumber  \\
%	& =  \int_\Theta  h(\mu_{\theta_0} \mid \mu_\theta  ) \, d \Pi_0 (\theta) \nonumber  \\
%	& = - h( \mu_{\theta_0}) -  \int_{\Theta} \Big(\int_\Omega \log J_\theta \, d  \mu_{\theta_0}\Big) \,d \Pi_0 (\theta),
%	\label{conseq:vp}
%\end{align}
%for $\mu_{\theta_0}$-almost every $y$. 
%\color{blue}
\begin{align}
\limsup_{n \to \infty} \frac{1}{n} \log Z_n(y) \,\, &  =  \limsup_{n \to \infty} \frac{1}{n} \log \int_\Theta  \frac{\mu_{\theta}
	\,(\,C_n(y) \,)}{\mu_{\theta_0} \,(\,C_n(y)\,) }   d \Pi_0 (\theta) \nonumber \\
	& \geqslant     \limsup_{n \to \infty} \frac{1}{n} \int_\Theta \log  \frac{\mu_{\theta}  (C_n(y) )}{\mu_{\theta_0} (C_n(y)) }
	%d \mu_{\theta_0}  (y)
	\, d \Pi_0 (\theta) \nonumber  \\
	& = - \int_\Theta  h(\mu_{\theta_0} \mid \mu_{\theta}  ) \, d \Pi_0 (\theta) \nonumber  \\
	& =  \int_{\Theta} \Big[\,  h( \mu_{\theta_0}) +   \int_\Omega \log J_{\theta} \, d  \mu_{\theta_0} \,\Big] \,d \Pi_0 (\theta)
	\label{conseq:vp}
\end{align}
for $\mu_{\theta_0}$-almost every $y$. 
On this context of direct observation we are interested in estimating the family of {\em a posteriori measures}
\begin{equation} \label{frt}
{\Pi_n (E \mid y)\,=\, \frac{ \int_E \mu_\theta (C_n (y) ) \,d \Pi_0 (\theta) }{ \int_\Theta \mu_\theta (C_n (y) ) \, d \Pi_0 (\theta)},}
\end{equation}
%\marginpar{\tiny is this expression OK? It seems it is the second one, am I wrong?}
on Borel sets $E \subset \Theta$ which \emph{do not contain $\theta_0$} and  $y\in \Omega$ is a point chosen according to $\mu_{\theta_0}$.
An equivalent form of (\ref{frt})  which may be useful  is
\begin{equation} \label{frt1}
\Pi_n (E \mid y)\,=\, \frac{ \int_E \,  \frac{\mu_{\theta}  \,(\,C_n(y) \,)}{\mu_{\theta_0} \,(\,C_n(y)\,) }     \,d \Pi_0 (\theta)}{ \int_\Theta \,  \frac{\mu_{\theta}  \,(\,C_n(y) \,)}{\mu_{\theta_0} \,(\,C_n(y)\,) }     \,d \Pi_0 (\theta)}.
\end{equation}
Actually, given such kind of $E \subset \Theta$, one can ask wether
the limit %there exists $u>0$, such that for $\mu_{\theta_0}$-almost every $y$ there exists the limit
\begin{equation} \label{yrt}
      \lim_{n \to \infty} \Pi_n (E \mid y)
      =  \lim_{n \to \infty} \frac{1}{n} \log    \frac{ \int_E \mu_\theta (C_n (y) ) d \Pi_0 (\theta) }{ \int_\Theta \mu_\theta (C_n (y) ) d \Pi_0 (\theta)}  \end{equation}
exists for $\mu_{\theta_0}$-almost every $y$.
The following result gives an affirmative answer to this question.

\begin{maintheorem} \label{thm:main}
In the previous context,
\begin{equation*}
      \lim_{n \to \infty} \Pi_n (\cdot \mid y)  = \delta_{\theta_0}, \quad \text{for $\mu_{\theta_0}$-a.e. $y\in \Omega$}.
  \end{equation*}
Moreover the convergence is exponentially fast: for every $\delta>0$ there exists a constant $c_\delta>0$ so that the ball $B_\delta$
of radius $\delta$ around $\theta_0$ satisfies
$
| \Pi_n (B_\delta \mid y)-1 |
	\leqslant
		 e^{ - c_\delta\,n}
$ for every large $n\geqslant 1$.
\end{maintheorem}

%
%\marginpar{\tiny maybe include \eqref{yrt} in the statement of the main theorem}
%

\medskip
The previous result guarantees that the parameter $\theta_0$, or equivalently the sampling measure $\mu_{\theta_0}$, is identified as the limit of the Bayesian inference process determined by the loss function \eqref{Chalk}. This result arises as a consequence of the quantitative estimates in Theorem~\ref{mmai},
given in the proofs section below.
The direct observation of Gibbs measures was also considered in \cite[Section 2.1]{Nobel1} although
with a different approach. For a parameterized family of loss functions
of the form $\beta \cdot \ell_n(\theta, x,y)$ it is also analyzed  in section 3.7 of \cite{Nobel1} the zero temperature limit (ground states). This is a topic  which can be associated to
ergodic optimization.
%\textcolor{blue}{Changing the metric on $\Omega$ we can mimic MP maps... weak Gibbs measures...}
%Our results are in some sense related to the so called Maximum Likelihood Identiﬁcation Procedure of \cite{CGL,Cha1}.
Our results are related in some sense to the so called Maximum Likelihood
Identification described \cite{CGL,Cessac,Cha2,Cha,Cha1}

\medskip
The previous context fits in the wider scope of non-additive thermodynamic formalism, using almost-additive sequences of continuous functions 
(see Subsection~\ref{subsec:almostadditive} for the definition).  Indeed,  the loss functions
$(\ell_n)_{n\geqslant 1}$ described in ~\eqref{Chalk} form an almost-additive family %on the $y$-variable 
(cf. Definition~\ref{def.almost.additive} and Lemma~\ref{le:ln-mu}). Furthermore, we will consider loss functions $\ell_n :\Theta \times \Omega \times Y \to \mathbb R$ 
which form an almost-additive sequence of continuous functions, and for which one can write
\begin{equation}\label{eq:lossphi0}
\ell_n(\theta, x,y) =  - \varphi_n(\theta, x,y),
\end{equation}
where $\varphi_n :\Theta \times \Omega \times Y \to \mathbb R_+$ are continuous observables satisfying:
\begin{enumerate}
%\item[(A1)] for each $\theta\in \Theta$ and $x\in \Omega$ there exists a constant $K_{\theta,x}>0$ so that, for every $y\in Y$,
%\begin{equation*}
%\varphi_n(\theta,x,y) + \varphi_m(\theta,x,T^n(y)) - K_{\theta,x}
%	\leqslant \varphi_{m+n}(\theta,x,y)
%	\leqslant \varphi_n(\theta,x,y) + \varphi_m(\theta,x,T^n(y)) + K_{\theta,x};
%\end{equation*}
%\item[(A2)] $\int e^{K_{\theta,x}} d\mu_{\theta}(x)<\infty$ for every $\theta\in \Theta$;
\item[(A1)] for $\nu$-almost every $y\in Y$ the following limit exists
	$$
	\Gamma^y(\theta):=\lim_{n\to\infty} \frac1n \log \int_\Omega e^{\varphi_n(\theta, x,y)} \, d\mu_\theta(x),
	$$
\item[(A2)] $\Theta\ni\theta \mapsto \Gamma^y(\theta)$ is upper semicontinuous.
\end{enumerate}
Given $y\in Y$ and the loss functions $\ell_n$ satisfying (A1)-(A2), the \emph{a posteriori} measures are
\begin{equation} \label{frtnew2}
\Pi_n (E \mid y)\,=\, \frac{ \int_E  \int_\Omega e^{\varphi_n(\theta,x,y)} \, d\mu_\theta(x) \, d \Pi_0 (\theta)}{ \int_\Theta  \int_\Omega e^{\varphi_n(\theta,x,y)} \, d\mu_\theta(x) \, d \Pi_0 (\theta)}.
\end{equation}

\medskip
\begin{remark}
\emph{
The expression appearing in assumption (A1), which resembles the logarithm of the moment generating function for i.i.d. random variables, is in special cases referred to as the free energy function. Consider the special case where $T=\sigma$ is the shift, $\nu$ is an equilibrium state with respect to a Lipschitz continuous potential $\psi$ and $\varphi_n(\theta, x,y) =\varphi_{n,1}(\theta, x) + \varphi_{n,2}(\theta,y)$, where $\varphi_{n,1}(\theta, x)=\sum_{j=0}^{n-1} \phi_\theta \circ \sigma^j(x)$, $\phi_\theta:\Omega\to\mathbb R$ is Lipschitz continuous and $(\varphi_{n,2}(\theta, \cdot))_{n\geqslant 1}$ is sub-additive.  Then using the fact that the pressure function defined over the space of Lipschitz continuous observables is Gateaux differentiable and the sub-additive ergodic theorem one obtains that
\begin{align*}
\Gamma^y(\theta)
%\lim_{n\to\infty} \frac1n \log \int_\Omega e^{\varphi_n(\theta, x,y)} \, d\mu_\theta(x)
	& = \lim_{n\to\infty} \frac1n \log \int_\Omega e^{\sum_{j=0}^{n-1} \phi_\theta (\sigma^j(x))} \, d\mu_\theta(x)
	+ \inf_{n\geqslant 1} \frac1n \int \varphi_{n,2}(\theta, \cdot) \, d\nu \\
	& = P_{\text{top}}(\sigma, \log J_\theta + \phi_\theta) - P_{\text{top}}(\sigma, \log J_\theta) 	+ \inf_{n\geqslant 1} \frac1n \int \varphi_{n,2}(\theta, \cdot) \, d\nu \\
	& = P_{\text{top}}(\sigma, \log J_\theta + \phi_\theta) + \inf_{n\geqslant 1} \frac1n \int \varphi_{n,2}(\theta, \cdot) \, d\nu ,
\end{align*}
for $\nu$-almost every $y\in \Omega$, hence it is independent of $y$. We refer the reader to Subsection~\ref{subsec:almostadditive} for the concept of topological pressure and further information.
}
\end{remark}

%which are independent of $\theta$ and $x$. As this gives no information on the parameter space,
%the \emph{a posteriori} measures are
%\begin{equation} \label{frtnew}
%\Pi_n (E \mid y)\,=\, \frac{ \int_E \varphi_n(y) \,d \Pi_0 (\theta) }{ \int_\Theta \varphi_n(y) \, d \Pi_0 (\theta)}
%	= \Pi_0 (E)
%\end{equation}
%In the case of the loss functions

The following result guarantees that the previous Bayesian inference procedure accumulates on the set of probability
measures on the parameter space $\Theta$ that maximize the free energy function $\Gamma^y$.
By assumption (A2) the set $\text{argmax} \,\Gamma^y:=\{\theta_0\in \Theta\colon \Gamma^y(\theta) \leqslant \Gamma^y(\theta_0), \forall \theta\in \Theta\}$ is non-empty. Then we prove the following:

\begin{maintheorem} \label{thm:main2}
Assume $\ell_n$ is a loss function of the form ~\eqref{eq:lossphi0} satisfying assumptions (A1)-(A2). There exists a full
$\nu$-measure subset $Y'\subset Y$ so that, for any $\delta>0$ and $y\in Y'$,
\begin{equation}\label{eq:thm:main2}
      \lim_{n \to \infty} \Pi_n (\Theta\setminus B_\delta^y \mid y)  = 0
      	\quad \text{where} \quad B_\delta^y=\{\theta \in \Theta \colon d_\Theta\big(\,\theta, \text{argmax} \,
	%_{\substack{\theta\in \Theta}}
	\Gamma^y\,\big)>\delta \}
  \end{equation}
is the open $\delta$-neighborhood of the maximality locus of $\Gamma^y$. In particular,  if $y\in Y'$
is such that $\Theta\ni\theta\mapsto \Gamma^y(\theta)$ has a unique point of maximum $\theta_0^y\in \Theta$
then
$
\lim_{n\to\infty} \Pi_n(\cdot\mid y)=\delta_{\theta^y_0}.
$
\end{maintheorem}

\medskip
Finally, inspired by the log-likelihood estimators in the context of Bayesian statistics it is also natural to consider
the loss functions $\ell_n :\Theta \times X \times Y \to \mathbb R$ defined  by
\begin{equation}\label{eq:lossphi}
\ell_n(\theta, x,y) = -\log \varphi_n(\theta, x,y)
\end{equation}
associated to an almost additive sequence $\Phi=(\varphi_n)_{n\geqslant 1}$
of continuous observables $\varphi_n :\Theta \times X \times Y \to \mathbb R_+$
satisfying
\begin{enumerate}
\item[(H1)] for each $\theta\in \Theta$ and $x\in X$ there exists a constant $K_{\theta,x}>0$ so that, for every $y\in Y$,
\begin{equation*}
\varphi_n(\theta,x,y) + \varphi_m(\theta,x,T^n(y)) - K_{\theta,x}
	\leqslant \varphi_{m+n}(\theta,x,y)
	\leqslant \varphi_n(\theta,x,y) + \varphi_m(\theta,x,T^n(y)) + K_{\theta,x}
\end{equation*}
\item[(H2)] $\int K_{\theta,x} d\mu_{\theta}(x)<\infty$ for every $\theta\in \Theta$.
\end{enumerate}
%\begin{enumerate}
%\item[(\ref{eq:lossphi}.1)] for each $\theta\in \Theta$ and $x\in X$ there exists a constant $K_{\theta,x}>0$ so that, for every $y\in Y$,
%\begin{equation*}
%\varphi_n(\theta,x,y) + \varphi_m(\theta,x,T^n(y)) - K_{\theta,x}
%	\leqslant \varphi_{m+n}(\theta,x,y)
%	\leqslant \varphi_n(\theta,x,y) + \varphi_m(\theta,x,T^n(y)) + K_{\theta,x}
%\end{equation*}
%\item[(\ref{eq:lossphi}.2)] $\int K_{\theta,x} d\mu_{\theta}(x)<\infty$ for every $\theta\in \Theta$.
%\end{enumerate}
In this context, the loss functions
induce the \emph{a posteriori} measures
\begin{equation} \label{frtnew2}
\Pi_n (E \mid y)\,=\, \frac{ \int_E  \psi_n(\theta,y)\, d \Pi_0 (\theta)}{ \int_\Theta  \psi_n(\theta,y) \, d \Pi_0 (\theta)},
\quad\text{where}\quad \psi_n(\theta,y)=\int_\Omega \varphi_n(\theta,x,y) \, d\mu_\theta(x).
\end{equation}
%
%\marginpar{\tiny \color{red} some extra \\ assumption may be needed! \\
%Example 2.1 shows \\convergence
%but not\\  for Dirac whenever \\ dependence is only \\ on one variable}
%
Therefore, even though the loss functions are not almost-additive, due to the logarithmic term, we have the following result for the latter
non-additive loss functions:

\begin{maintheorem} \label{thm:main3}
Assume that the loss function of the form ~\eqref{eq:lossphi} satisfies assumptions (H1)-(H2) above.
There exists a non-negative function $\psi_*:\Theta \to \mathbb R_+$ (depending on $\Psi^\theta=(\psi_n(\theta,\cdot))_{n\ge 1}$)
so that for $\nu$-almost every $y\in Y$
the \emph{a posteriori} measures $(\Pi_n (\cdot \mid y))_{n\geqslant 1}$ are convergent and
\begin{equation*}
%\lim_{n\to\infty}
\Pi_n (\cdot \mid y)\,=\, \frac{ \int_{\cdot}  \psi_n(\theta,y)\, d \Pi_0 (\theta)}{ \int_\Theta  \psi_n(\theta,y) \, d \Pi_0 (\theta)} 	
	\longrightarrow \Pi_*(\cdot):=\frac{(\psi_* \Pi_0)(\cdot)}{(\psi_* \Pi_0)(\Theta)}
\end{equation*}
as $n\to\infty$.
Moreover, if $T=\sigma$ is a subshift of finite type, $\nu\in \mathcal M_\sigma(\Omega)$ is a Gibbs measure with respect to
a Lipschitz continuous potential %, $\Psi^\theta=(\psi_n(\theta,\cdot))_{n\ge 1}$
and $\inf_{\theta\in \Theta}\psi_*(\theta)>0$ then for each $g\in C(\Theta, \mathbb R)$ there exists $c>0$ so that
\begin{align}
\limsup_{n \to \infty} \frac{1}{n}\,\,  & \log \, \nu ( \{\,y \in \Omega :
	\Big|\int g\, d\Pi_n (\cdot \mid y) - \int g\, d\Pi_*\Big|\geqslant\delta  \}) \nonumber \\
		& \leqslant  \sup_{\theta\in \Theta}
		\sup_{\{\eta\colon |{\mathcal F}(\eta,\Psi^\theta) -\psi_*(\theta)|\geqslant c \delta \}}
		\Big\{-P(\sigma,\varphi)+h_\eta(\sigma) + \int \varphi\, d\eta \Big\}, \label{eq:lde}
\end{align}
where  ${\mathcal F} (\eta, \Psi^\theta) := \lim_{n \to \infty} \frac{1}{n} \int \psi_n(\theta,\cdot) \, d \eta.$
If, additionally, the map $\Theta\ni \theta\mapsto\mathcal F(\eta,\Psi^\theta)$ is continuous for each
$\eta\in \mathcal M_\sigma(\Omega)$ then the right hand-side in ~\eqref{eq:lde} is strictly negative.
\end{maintheorem}
\color{black}

\color{black}

The previous theorem ensures that, in the context of loss functions of the form ~\eqref{eq:lossphi} satisfying properties (H1) and (H2) above, the a posteriori measures do converge exponentially fast to a probability measure  on the parameter space which is typically fully supported.
We refer the reader to Example~\ref{ex0} for more details in the special case the loss function depends exclusively on one parameter.

\begin{remark}
\emph{
For completeness, let us mention that the results by Kifer \cite{Kif} suggest that level-2 large deviations estimates (ie, the rate of convergence of $\Pi_n (\cdot \mid y)$
to $\Pi_*$ on the space of probability measures on $\Theta$)
%$$
%\limsup_{n\to\infty}
%	\frac1n \log
%	\nu\Big(y\in Y \colon d_1\Big(\Pi_n (\cdot \mid y), \Pi_* \Big)>\delta\Big)
%$$
%(where $d_1$ denotes a metric compatible with the weak$^*$ topology)
are likely to hold under the assumption that the limit
$
\lim_{n\to\infty} \frac1n \log \int e^{\varphi_n} \,d\nu
$
exists for all almost-additive sequences $\Phi=(\varphi_n)_{n\geqslant 1}$ of continuous observables and defines a non-additive free energy function which is related to the non-additive topological pressure. This extrapolates the scope of our interest here.
}
\end{remark}

\section{Examples}\label{sec:examples}

In what follows we give some examples which illustrate the intuition and utility of the Bayesian inference
and also the meaning of the \emph{a priori} measures. 

\begin{example} \label{exlo}  The space of all stationary Markov probability measures $\mu$ in $\Omega=\{1,2\}^\mathbb{N}$ is described
by the space of column stochastic transition matrices $P$ with all positive entries. These matrices $P$ can be parameterized by the open
square $\Theta =(0,1) \times (0,1)$ through the parameterization
$$
M_{(a,b)}= \left( \begin{array}{cc} P_{11} & P_{12} \\ P_{21}& P_{22} \end{array} \right)=\left( \begin{array}{cc} a &
 1-b \\ 1-a& b \end{array} \right),
	 \qquad (a,b) \in (0,1) \times (0,1).
 $$
 In this case the associated normalized Jacobian $J_{(a,b)}(w)$  has  constant value  on cylinders of size  two. More precisely,  for $w$ on the cylinder
$[i, j]\subset \Omega$ we get $J = \frac{\pi_i \,P_{i,j}}{\pi_j}$,  where $(\pi_1,\pi_2)$ is the initial invariant probability vector. For each value $(a,b)$ denote by $\mu_{(a,b)}$ the stationary Markov probability measure  associated to the stochastic matrix $ M_{(a,b)}$. In this case we get that $ h(\mu_{(a,b)}) \,+\, \int \log J_{(a,b)} \, d \mu_{(a,b)} =0$ and $\mathcal{L}_{\log J_{(a,b)}}^* (\mu_{(a,b)})= \mu_{(a,b)}$ (see \cite{L3,Sp}).
We refer the reader to \cite{Douc,Douc1,Douc2,Tan} for applications of the use of the maximum likelihood estimator in this context of Markov probability measures.
One possibility would be to take the probability measure  $\Pi_0$ on the $\Theta$ space as the Lebesgue probability measure  on $(0,1) \times (0,1).$ Different choices of loss functions would lead to different solutions for the claim of Theorem \ref{thm:main2}.
\end{example}

The first of the following examples are very simple and illustrate some trivial contexts.
Whenever the parameter space $\Theta$ (or $Y$) is a singleton the Bayesian inference is trivial, hence it carries no information. The first example
we shall consider is when the loss function depends exclusively on a single variable. Nevertheless, as loss functions are non-additive, these results could not be
handled with the previous literature in the subject.

\begin{example} \label{ex0}
Assume that $\Theta\subset \mathbb R^d$ is a compact set, $Y=\Omega$ and  $T=\sigma: \Omega \to \Omega$ is a subshift of finite
type. In the case that the loss functions $\ell_n \colon \Theta \times \Omega \times Y \to\mathbb R$ are generated by an almost-additive sequence of continuous observables $\Phi=(\varphi_n)_{n\geqslant 1}$ by
$
\ell_n(\theta, x,y) = -\log \varphi_n(y)
$
which are independent of $\theta$ and $x$, the loss function gives no information on the parameter space. For that reason
it is natural that the \emph{a posteriori} measures are
\begin{equation} \label{frtnew}
\Pi_n (E \mid y)\,=\, \frac{ \int_E \varphi_n(y) \,d \Pi_0 (\theta) }{ \int_\Theta \varphi_n(y) \, d \Pi_0 (\theta)}
	= \Pi_0 (E)
\end{equation}
for every sampling $y, T(y), \dots, T^{n-1}(y) \in Y$.

Now, assuming alternatively that the loss function is given by
$
\ell_n(\theta, x,y) = -\log \varphi_n(\theta),
$
which is independent on both $x$ and $y$, a simple computation shows that
\begin{equation} \label{frtnew-theta}
\Pi_n (E \mid y)\,=\, \frac{ \int_E \varphi_n(\theta) \,d \Pi_0 (\theta) }{ \int_\Theta \varphi_n(\theta) \, d \Pi_0 (\theta)}.
\end{equation}
In this case the loss function neglects the observable dynamical system $T$, hence the a posteriori measures are
independent of the sampling. Yet, as the family $\Phi$ is almost-additive it is easy to check that there exists $C>0$ so that
$\{\varphi_n+C\}_{n\geqslant 1}$ is sub-additive. In particular, a simple application of Fekete's lemma (cf. Lemma~\ref{le:Fekete}) ensures
that the limit $\lim_{n\to\infty}\frac{\phi_n(\theta)}n$ does exists and coincides with
$\phi_*(\theta):=\inf_{n\geqslant 1} \frac{\phi_n(\theta)}n$, for every $\theta\in \Theta$. In consequence,
\begin{equation} \label{le:Fekete}
\lim_{n\to\infty} \Pi_n (E \mid y)= \lim_{n\to\infty} \frac{ \int_E \frac{\varphi_n(\theta)}n \,d \Pi_0 (\theta) }{ \int_\Theta \frac{\varphi_n(\theta)}n \, d \Pi_0 (\theta)}
	= \Pi(E):= \frac{ \int_E \varphi_*(\theta) \,d \Pi_0 (\theta) }{ \int_\Theta \varphi_*(\theta) \, d \Pi_0 (\theta)},
\end{equation}
independently of the sampling $y$. In particular the limit measure $\Pi$ is fully supported on $\Theta$ if and only
if $\varphi_*(\theta)>0$ for every $\theta\in \Theta$.

Finally, for each $n\geqslant 1$ and almost-additive sequence
of continuous observables $\Phi=(\varphi_n)_{n\geqslant 1}$ on $X$, consider the loss function
%$$
%\ell_n(\theta, x,y) = -\log \varphi_n(y)
%$$
%which are independent of $\theta$ and $x$. As this gives no information on the parameter space,
%the \emph{a posteriori} measures are
%\begin{equation} \label{frtnew}
%\Pi_n (E \mid y)\,=\, \frac{ \int_E \varphi_n(y) \,d \Pi_0 (\theta) }{ \int_\Theta \varphi_n(y) \, d \Pi_0 (\theta)}
%	= \Pi_0 (E)
%\end{equation}
%In the case of the loss functions
$$
\ell_n(\theta, x,y) = -\log \varphi_n(x),
$$
In this case a simple computation shows that one obtains \emph{a posteriori} measures
\begin{equation} \label{frtnew2}
\Pi_n (E \mid y)\,=\, \frac{ \int_E  \psi_n(\theta)\, d \Pi_0 (\theta)}{ \int_\Theta  \psi_n(\theta) \, d \Pi_0 (\theta)},
\end{equation}
where the sequence $\psi_n(\theta)=\int_\Omega \varphi_n(x) \, d\mu_\theta(x)$ is almost additive. Indeed, the
$\sigma$-invariance of $\mu_\theta$ and the almost-additivity condition
$
\varphi_{n}(x) + \varphi_{m}(\sigma^n(x)) - C
	\leqslant \varphi_{m+n}(x)
	\leqslant \varphi_{n}(x) + \varphi_{m}(\sigma^n(x)) +C
$
ensures that
$
\psi_{n}(\theta) + \psi_{m}(\theta) - C
	\leqslant \psi_{m+n}(\theta)
	\leqslant  \psi_{n}(\theta) + \psi_{m}(\theta) + C
$
for every $m,n\geqslant 1$ and $\theta\in \Omega$.
Hence, even though the feed of information is given
through the $x$-variable, the \emph{a posteriori} measures are of the form ~\eqref{frtnew}, and
their convergence is described by Lemma~\ref{le:Fekete}.
In particular, this example shows that the situation is much simpler to describe when the loss functions depend exclusively
on a single variable.

\end{example}

In the following two simple examples, we will make explicit computations on the limit of posterior
distributions which shows that the assumption (A) on the space of parameters and a priori distribution cannot be removed. 
In particular, these will show that the posterior distributions $\Pi_n(\cdot\mid y)$ may converge but 
not for a Dirac measure on the parameter $\theta_0$ corresponding to the measure with respect to which the sampling occurs.

\begin{example} \label{ex01}
Set $\Theta=\{-1,1\}$, $T=\sigma: \{0,1\}^{\mathbb N} \to \{0,1\}^{\mathbb N}$ be the full shift and let $\mathbb B(a,b)$ denote the Bernoulli measure with $\nu[0]=a$ and $\nu[1]=b$, for $a+b=1$, $0<a<1$.  If $\phi: \{0,1\}^{\mathbb N} \to\mathbb R$ is a locally constant normalized potential so that $\phi\mid_{[0]}=c<0$ then it is not hard to deduce (see e.g. \cite{BowenLNM}) that $\phi\mid_{[1]}=\log(1-e^c)$ and the unique equilibrium state for $\sigma$ with respect to $\phi$ is the probability measure  $\mathbb B(e^{c}, 1-e^c)$.
Assume that $\mu_{-1}=\mathbb B(\frac13,\frac23)$ and $\mu_{1}=\mathbb B(\frac23,\frac13)$ which are the unique equilibrium states for the potentials
$$
\phi_{-1}(x):=
\begin{cases}
-\log 3, x\in[0] \\
-\log \frac32, x\in[1]
\end{cases}
	\quad\text{and}\quad
\phi_{1}(x):=
\begin{cases}
-\log \frac32, x\in[0] \\
-\log 3, x\in[1],
\end{cases}
$$
respectively. Take $\Pi_0=\frac12\delta_{-1}+\frac12\delta_{1}$ and $\nu=\mathbb B(\frac12,\frac12)$ and notice that $\nu$ does not belong to the family $(\mu_\theta)_{\theta}$.
On the context of direct observation we are interested in describing the {\em a posteriori} measures
\begin{equation*}
{\Pi_n (E \mid y)\,=\, \frac{ \int_E \mu_\theta (C_n (y) ) \,d \Pi_0 (\theta) }{ \int_\Theta \mu_\theta (C_n (y) ) \, d \Pi_0 (\theta)},}
\end{equation*}
for samplings over $\nu$. The ergodic theorem ensures that
$$
\lim_{n\to\infty}\frac1n\#\{0\leqslant j \leqslant n-1 \colon \sigma^j(y) \in [0]\}
	= \lim_{n\to\infty}\frac1n\#\{0\leqslant j \leqslant n-1 \colon \sigma^j(y) \in [1]\}
	= \frac12
$$
for $\nu$-almost every $y$. The Bernoulli property of $\mu_{\pm1}$ then implies that, for $\nu$-a.e.
$y$,
%$\mu_{\pm1} (C_n (y) )/(\sqrt{2}/3)^n \to 1$ as $n\to\infty$ for $\nu$-almost every $y$.
$$
\frac{\mu_{1} (C_n (y) )}{\mu_{-1} (C_n (y) )} \to 1 \quad\text{as $n\to\infty$}
$$
and, consequently, the sequence of probability measures $\Pi_n (\cdot \mid y)$ on $\{-1,1\}$ is convergent as
\begin{equation*}
\lim_{n\to\infty }\Pi_n (\{\pm1\} \mid y)
	\,=\, \lim_{n\to\infty } \frac{ \mu_{\pm 1} (C_n (y) ) }{\mu_{-1} (C_n (y) ) + \mu_{1} (C_n (y) )} = \frac12.
\end{equation*}
In other words, $\lim_{n\to\infty }\Pi_n (\cdot \mid y)=\frac12\delta_{-1}+\frac12\delta_{1}=\Pi_0$. This convergence reflects the fact that $\int \phi_{-1}\, d\nu=\int \phi_{1}\, d\nu$. Finally, it is not hard to check that for any a priori measure
$\Pi_0=\alpha\delta_{-1}+(1-\alpha)\delta_{1}$ for some $0<\alpha<1$ it still holds that
$\lim_{n\to\infty }\Pi_n (\cdot \mid y)=\Pi_0$.
\end{example}

\begin{example} \label{ex001}
In the context of Example~\ref{ex01}, assume that the sampling is done with respect to a non-symmetric Bernoulli measure
$\hat \nu=\mathbb B(\alpha,1-\alpha)$ for some $0<\alpha<\frac12$. The ergodic theorem guarantees that, for $\hat \nu$-a.e. $y$,
 $$
\frac{\mu_{1} (C_n (y) )}{\frac{2^{\alpha n}}{3^n}} \to 1
	\quad \text{and}\quad
\frac{\mu_{-1} (C_n (y) )}{\frac{2^{(1-\alpha) n}}{3^n}} \to 1
	\quad\text{as $n\to\infty$}
$$
and, consequently, $\mu_{1} (C_n (y) ) / \mu_{-1} (C_n (y) ) \to 0$ as $n\to\infty$.
 Altogether we get
 \begin{equation*}
\lim_{n\to\infty }\Pi_n (\{1\} \mid y)
	\,=\, \lim_{n\to\infty } \frac{ \mu_{ 1} (C_n (y) ) }{\mu_{-1} (C_n (y) ) + \mu_{1} (C_n (y) )}
	\,=\, \lim_{n\to\infty } \frac{ \frac{ \mu_{ 1} (C_n (y) ) }{\mu_{-1} (C_n (y) )} }{1+\frac{ \mu_{ 1} (C_n (y) ) }{\mu_{-1} (C_n (y) )}} =0
\end{equation*}
 and $\lim_{n\to\infty }\Pi_n (\{-1\} \mid y)=1$. In other words, $\lim_{n\to\infty }\Pi_n (\cdot\mid y)=\delta_{-1}$
for $\hat\nu$-almost every $y$, which reflects the fact that $\int \phi_{1}\, d\hat\nu<\int \phi_{-1}\, d\hat\nu$.
\end{example}

%\begin{example} \label{ex1}
%As an example one can take $\Theta=\{0,1\}$ and
%$$\Pi_0= \pi_0^0 \delta_0 + \pi_1^0 \delta_1.$$

%Take $\log J_0$ and $\log J_1$ two normalized Lipschitz potentials
%and $f(1) = \log J_1$ and $f(0)= \log J_0.$

\begin{example} \label{ex2}
Take $\Theta=[0,1]$ and let $\Pi_0$ be the Lebesgue measure. Take $\log J_0$ and $\log J_1$ two normalized Lipschitz continuous Jacobians, where $J_0,J_1 : \{1,2,...,q\} ^\mathbb{N} \to \mathbb{R}_+,$ and consider the family of Lipschitz continuous
potentials
$$
f_\theta =  \log J_\theta :=\log (\theta  J_1 + (1-\theta)  J_0), \qquad \theta\in [0,1].
$$
%In this case $\Pi_0(\theta)$ is $d \theta$.
For each $\theta\in [0,1]$ let $\mu_\theta$ be the unique Gibbs measure associated to the Lipschitz continuous potential $f_\theta$ (see also section 6 in  \cite{FLL} for a related work).
Assume further that the observed probability measure  associated to the sampling is $\nu=\mu_{\theta_0}$
%$\nu=  \log (\theta_0  J_1 + (1-\theta_0)  J_0)=\log J_{\theta_0}$,
 for some  $\theta_0\in [0,1]$.
 %
 %We do not know in advance what is $\log J_{\theta_0}$.
The probability measure  $\Pi_0$ describes our ignorance of the exact value $\theta_0$ among all possible choices $\theta\in [0,1]$.
For each $n \in \mathbb{N}$  consider a continuous  \emph{loss function}
$
\ell_n : \Theta    \times \Omega \times       Y \to \mathbb{R}
$
expressed as
$$\ell_n((a,b), x, y) =-
\sum_{j=0}^{n-1} \log J_{\theta} (\sigma^j (x)) +  \sum_{j=0}^{n-1} \log J_{\theta_0}(\sigma^j (y)) + \, \theta \log \theta -\theta \log \theta_0\,.$$
Similar expressions are often referred as cross-entropy loss functions.
%In this case, the term appearing in the Gibbs posterior distribution involves
%$$
%e^{- \ell_n} =  \frac{\Pi_{j=0}^{n-1} J_{\theta} (\sigma^j (x)) }{\Pi_{j=0}^{n-1} J_{\theta_0 } (\sigma^j (y)) }  \frac{\theta^\theta }{ \theta_0^\theta}.
%$$
By compactness of the parameter space $\Theta$ we conclude that the third and fourth expressions above are uniformly bounded, hence $(\ell_n)_{n\geqslant1}$ forms an almost-additive family on the $y$-variable, hence it fits in the
context of Theorem~\ref{thm:main2}.
%
%\marginpar{\tiny the goal/why it works seems incomplete: is it minimization or detection of $\theta_0$? why
%any of these should work?}
%
%We want to minimize $- \ell_n.$
% Hence, one would be interested in finding an iterative process $\Pi_n$  such that $\Pi_n$ converges to the probability $\mu_{\theta_0}$ as  when $n$ tends to infinity.
%As usual we wander about  $\mu_{\theta_0}$ almost $y$
%$$
% Z_n(y)=  \int \int_\Theta e^{ - \, \ell_n (\theta, x,y)}d \Pi_0 (\theta)\,d \mu_{\theta} (x).
% $$
%Hence, one would be interested in finding an iterative process $\Pi_n$  such that $\Pi_n$ converges to the probability $\mu_{\theta_0}$ as  when $n$ tends to infinity.
In particular we conclude that the \emph{a posteriori} measures $\Pi_n(\cdot\mid y)$ converge to the probability measure 
$\mu_{\theta_0}$ as $n$ tends to infinity, for $\mu_{\theta_0}$-almost every $y$.
Alternatively, consider the continuous  \emph{loss function} $
\ell_n : \Theta    \times \Omega \times       Y \to \mathbb{R}
$
given by
$$
\ell_n((a,b), x, y) = - \sum_{j=0}^{n-1} \log J_{\theta} (\sigma^j (x))  +  \sum_{j=0}^{n-1} \log J_{\theta_0}(\sigma^j (y)) - \|\, \theta - \theta_0\,\|^2.
$$
 The minimization of $- \ell_n$ corresponds, in rough terms, to what is known in statistics as the minimization of the mean squared error  on the set of parameters. As the previous loss function is also almost-additive on the $y$-variable, Theorem~\ref{thm:main2} ensures that the corresponding \emph{a posteriori} measures $\Pi_n(\cdot\mid y)$ converge exponentially fast to the sampling probability measure  $\mu_{\theta_0}$ as $n$ tends to infinity, $\mu_{\theta_0}$-almost everywhere
(we refer the reader to \cite{Nobel1} where the methods which were developed there can provide an alternative argument leading to the same conclusion).
\end{example}

\begin{example} \label{ex03}
Let $\sigma:\{1,2\}^{\mathbb N} \to \{1,2\}^{\mathbb N}$ be the full shift and for each $\theta=(\theta_1, \theta_2)$ in the parameter space $\Theta:=[-\varepsilon, \varepsilon]^2$ let $\mu_\theta$
%(e.g. $\mathbb B(\theta_1+\theta_2, 1-\theta_1-\theta_2)$)
be a continuous family of Bernoulli measures. These are equilibrium states for a continuous family of potentials.
Consider also the locally constant linear cocycle $A_\theta: \Omega \to SL(2,\mathbb R)$
given by
$$
A_\theta\mid_{[i]}=
\begin{pmatrix}
\begin{array}{cc}
2 & 1  \\ 1 & 1
\end{array}
\end{pmatrix}
\,\cdot\,
\begin{pmatrix}
\begin{array}{cc}
\cos \theta_i & -\sin \theta_i  \\ \sin \theta_i & \cos \theta_i
\end{array}
\end{pmatrix},
\quad\text{for every $i=1, 2$. }
$$
Given $n\geqslant 1$ and $(x_1, \dots, x_n) \in \{1,2\}^n$, take the matrix product
$$A_\theta^{(n)}(x_1, \dots, x_n):= A_{\theta_{x_n}} \dots A_{\theta_{x_2}} A_{\theta_{x_1}}.
$$
The limit
%If $\varepsilon>0$ is chosen small the previous family of matrices preserve a constant cone field in $\mathbb R^2$, hence
%have a dominated splitting. Furthermore, if $\mu\in \mathcal M_\sigma(\Omega)$ is ergodic the Oseledets theorem ensures that for $\mu$-almost every $x\in \Omega$ there exists a cocycle invariant splitting
%$\mathbb R^2=E^+_{\theta,x} \oplus E^-_{\theta,x}$ so that the limit
$
\lambda_{\theta,i} := \lim_{n\to\infty} \frac1n \log \| A_\theta^{(n)}(x) \, v_i \|
$
is the largest Lyapunov exponent along the orbit of $x$, it is well defined for $\nu$-almost every $x$
and independs on the vector  $v_i\in E^i_{\theta,x} \setminus\{0\},$ $(i=\pm)$ (cf. Subsection~\ref{subsec:Lyapexp} for more details).
Somewhat dual to the context of joint spectral radius \cite{Bochi}, the problem here is the selection of a certain Gibbs measure from the information
on the norm of the products of matrices, along orbits of typical points. %, determined by the samplings.
More precisely,
%According to Kingman's sub-additive ergodic theorem the largest Lyapunov exponent is obtained by
%$$
%\lambda^+(A_\theta,\mu)=\inf_{n\geqslant 1} \frac1n \log \|A^{(n)}_{\theta}(x)\|,
%$$
%and since all matrices preserve a cone field then for each $\theta\in[-\varepsilon,\varepsilon]^p$ the sequence
%$(\log \|A^{(n)}_{\theta}(x)\|)_{n\geqslant 1}$ is almost-additive
%(cf.~\cite{Feng}). The same arguments as in Lemma~\ref{kkk1} yield the following
%
%
%\begin{lemma} \label{Lyapkkk1}
%For any fixed $y\in Y$ and any Borel set $E\subset \Theta$,  the family
%$$ \psi_{n}(y)= \psi_{n}^{E} (y)  = -\,\,\int  \mathbf 1_E (\theta)\, \log \|A^{(n)}_{\theta}(x)\| \, d \Pi_0 (\theta),$$
%$n \in \mathbb{N},$ is almost-additive.
%In particular, for each $\theta_0\in \Theta$ and $E\subset\Theta$, the family $\Psi^E := \{\Psi_n^E\}_n, $
%\begin{equation}\label{Lyapkkk2}
%y \to \Psi_{n}^{E} (y) := -\,\,\int_E\, \log \frac{\|A^{(n)}_{\theta}(x)\|}{ \|A^{(n)}_{\theta_0}(x)\| }  \, d \Pi_0 (\theta),
%\end{equation}
%is almost additive.
%\end{lemma}
%Moreover, 
take the loss function $\ell_n(\theta, x,y) = -\log \|A^{(n)}_{\theta}(x)\|$ and notice that, for $\nu$-almost every $y\in Y$ and every $\theta\in \Theta$,
\begin{align*}
%\Gamma^y(\theta):=
\int_\Omega e^{\varphi_{n+m}(\theta, x,y)} \, d\mu_\theta(x)
	& = \int_\Omega \|A^{(n+m)}_{\theta}(x)\| \, d\mu_\theta(x)
	=   \sum_{C_{n+m}(z)} \|A^{(n+m)}_{\theta}(z)\| \, \mu_\theta(C_{n+m}(z)) \\
	& \leqslant \sum_{C_n(z)} \|A^{(n)}_{\theta}(z)\|  \|A^{(m)}_{\theta}(\sigma^n(z))\| \, \mu_\theta(C_{n}(z)) \, \mu_\theta(C_{m}(\sigma^n(z))) \\
	& \leqslant \int_\Omega e^{\varphi_{n}(\theta, x,y)} \, d\mu_\theta(x) \cdot \int_\Omega e^{\varphi_{m}(\theta, x,y)} \, d\mu_\theta(x)
\end{align*}
for every $m,n\geqslant 1$, where we used that $\mu_\theta$ is a $\sigma$-invariant Bernoulli measure. In particular, Fekete's lemma implies that the following limit exists
\begin{align*}
%\Gamma^y(\theta):=
\lim_{n\to\infty} \frac1n \log \int_\Omega e^{\varphi_n(\theta, x,y)} \, d\mu_\theta(x)
	& = \inf_{n\geqslant 1} \frac1n \log \int_\Omega \|A^{(n)}_{\theta}(x)\| \, d\mu_\theta(x).
\end{align*}
exists and independs on $y$. As the right hand-side above is the infimum of continuous functions on the parameter $\theta$, the limit function $\Theta\ni\theta\mapsto \Gamma^y(\theta)$ is upper semicontinuous.
We remark that $\theta_0=(0,0)$ is the unique parameter for which the Lyapunov exponent is the largest possible (see Lemma~\ref{le:Furstenberg}). Hence,
as assumptions (A1) and (A2) are satisfied, Theorem~\ref{thm:main2} implies that
\begin{equation*}
\Pi_n (E \mid y)\,=\, \frac{ \int_E  \int \|A^{(n)}_{\theta}(x)\| d\mu_\theta(x) \,d \Pi_0 (\theta) }{  \int_\Theta
\int \|A^{(n)}_{\theta}(x)\| d\mu_\theta(x) \,d \Pi_0 (\theta)}
	\longrightarrow
	\begin{cases}
	1, \text{if} \; (0,0) \in E \\
	0, \text{otherwise}
	\end{cases}
\end{equation*}
for every measurable subset $E\subset \Theta$, In other words, the a posteriori measures converge to the Dirac measure
$\delta_{(0,0)}$. In particular, one has posterior consistency in the problem of determining the measure with largest Lyapunov exponent. 

Alternatively, taking the loss function $\ell_n(\theta, x,y) =-\varphi_n(\theta, x,y) = -\log \|A^{(n)}_{\theta}(y)\|$, note that
the \emph{a posteriori} measures are given by
\begin{equation*}
\Pi_n (E \mid y)\,=\, \frac{ \int_E  \|A^{(n)}_{\theta}(y)\| \,d \Pi_0 (\theta) }{  \int_\Theta  \|A^{(n)}_{\theta}(y)\| \,d \Pi_0 (\theta)}
\end{equation*}
and that, by the Oseledets theorem and the sub-additive ergodic theorem, the limit
$$
%:=
\lim_{n\to\infty} \frac1n \log \int_\Omega e^{\varphi_n(\theta, x,y)} \, d\mu_\theta(x)
	= \lim_{n\to\infty} \frac1n \log \|A^{(n)}_{\theta}(y)\|
	=  \inf_{n\geqslant 1} \frac1n \int \log \|A^{(n)}_{\theta}(\cdot)\| \,d\nu
$$
for $\nu$-almost every $y\in Y$. The map
$
\Theta\ni \theta\mapsto \Gamma^y(\theta) :=  \inf_{n\geqslant 1} \frac1n \int \log \|A^{(n)}_{\theta}(\cdot)\| \,d\nu
$
is upper semicontinuous because it is the infimum of continuous maps. In particular, Theorem~\ref{thm:main2} implies once
more that
for $\nu$-almost every $y\in Y$
\begin{equation*}
\lim_{n\to\infty}\Pi_n (\cdot \mid y)\,=\, \lim_{n\to\infty} \frac{ \int_\cdot  \|A^{(n)}_{\theta}(y)\| \,d \Pi_0 (\theta) }{  \int_\Theta  \|A^{(n)}_{\theta}(y)\| \,d \Pi_0 (\theta)} = \delta_{(0,0)}
\end{equation*}
\end{example}

\begin{example} \label{ex04}
In the context of Example~\ref{ex03}, noticing that all matrices are in $SL(2,\mathbb R)$ it makes sense to consider alternatively the loss function
$
\ell_n(\theta, x,y) = -\log \varphi_n(\theta,x,y)=- \log \log \|A^{(n)}_{\theta}(y)\|,
$
and to observe that $\varphi_n(\theta,x,y)$ is almost-additive, meaning it satisfies (H1)-(H2) with a constant K uniform on
$\theta$. The loss functions
induce the \emph{a posteriori} measures
\begin{equation} \label{frtnew2Lyap}
\Pi_n (E \mid y)\,=\, \frac{ \int_E   \int \log \|A^{(n)}_{\theta}(x)\| d\mu_\theta(x)\, d \Pi_0 (\theta)}{ \int_\Theta
\int \log \|A^{(n)}_{\theta}(x)\| d\mu_\theta(x)\, d \Pi_0 (\theta)}.
\end{equation}
A simple computation involving Fekete's lemma guarantees that, for each $\theta\in \Theta$, the annealed Lyapunov exponent
$$
\lambda(\theta):=\lim_{n\to\infty} \frac1n \int \log \|A^{(n)}_{\theta}(x)\| d\mu_\theta(x) =\inf_{n\geqslant 1} \frac1n \int \log \|A^{(n)}_{\theta}(x)\| d\mu_\theta(x) \geqslant 0
$$
does exist. Theorem~\ref{thm:main3} implies that the
\emph{a posteriori} measures ~\eqref{frtnew2Lyap} converge and
\begin{equation*}
\lim_{n\to\infty} \Pi_n (E \mid y)\,=\, \frac{\int_E \lambda(\theta) \, d\Pi_0(\theta)}{\int_\Omega \lambda(\theta) \, d\Pi_0(\theta)}
\end{equation*}
for every measurable subset $E\subset \Theta$. In particular the limit measure is absolutely continuous with respect to
the \emph{a priori} measure $\Pi_0$ and with density given by the normalized Lyapunov exponent function.
Moreover, the continuous dependence of the Lyapunov exponents with respect to the parameter $\theta$ implies
the exponential large deviations estimates in Theorem~\ref{thm:main3}.
\end{example}

\color{black}

\section{ Preliminaries}\label{se:prelim}

\subsection{Relative entropy}\label{subsecrelent}

Let us recall some relevant concepts of entropy in the context of shifts.
Given $x=  (x_1,x_2,...,x_k ,...)\in \Omega$  and $n\geqslant 1$,  recall
$$C_n (x)=\{ y\in \Omega \,|\,y_j=x_j, j=1,2,..,n\}$$ the \emph{$n$-cylinder} in $\Omega$ that contains the point $x$. %(according to  the  notation in \cite{Cha2} and \cite{VZ}).
The concept of relative entropy will play a key role in the analysis.
%
%\begin{definition}
Let $\phi \colon \Omega\to\mathbb R$ be a Lipschitz continuous potential and let $\mu_\phi$ be its unique Gibbs measure, thus ergodic. Following \cite[Section~3]{Cha1}, given an ergodic probability measure  $\mu \in \mathcal M_\sigma(\Omega)$ the limit
\begin{align} \label{Chacha}
h(\mu \mid \mu_\phi):=\lim_{n \to \infty} \frac{ 1 }{ n}  \log \Big(\frac{ \mu (   C_n (x))}{\mu_\phi  (   C_n (x) )} \Big)
	 %= \lim_{n \to \infty} \frac{ 1 }{ n}  \log \Big(\frac{\mu_\phi  \,(\,\overline{ x_1,x_2,...,x_n } \,)}{\mu \,(\,\overline{ x_1,x_2,...,x_n }\,) }\,\Big)
	 %\geqslant 0
\end{align}
exists and is non-negative for $\mu$-almost every $x=(x_1,x_2,...,x_n,..)\in \Omega$, and it is called the
\emph{relative entropy} of $\mu$ with respect to $\mu_\phi$. Notice that any two distinct ergodic probability measures are mutually singular,
hence no Radon-Nykodym derivative is well defined. 
In ~\eqref{Chacha}, a sequence of nested cylinder sets is used as an alternative to
compute relative entropy when  %gives an expression of the relative entropy of two probability measures which may be singular 
Radon-Nykodym derivatives are not well defined (see \cite{Cha1} for more details). Moreover,
\begin{equation}\label{eq:computable}
h(\mu \mid \mu_\phi)
	= P_{\text{top}}(\sigma,\phi) - \int \phi \, d \mu - h(\mu)
\end{equation}
and, by the variational principle and uniqueness of equilibrium states for Lipschitz continuous potentials, $h(\mu \mid \mu_\phi)=0$ if and only if $\mu=\mu_\phi$
(cf. Subsection 3.2 in \cite{Cha1}). Furthermore, if $\mu=\mu_\psi$ is a Gibbs measure then $h(\mu \mid \mu_\phi)=0$
if and only if $\phi$ and $\psi$ are cohomologous, i.e, if there exists a Lipschitz continuous function $u: \Omega\to\mathbb R$
so that $\phi=\psi + u\circ \sigma - u$. 
%Besides, every Lipschitz continuous potential $A$ is cohomologous to a normalized one, 
The relative entropy is also known as the Kullback-Leibler divergence.
For proofs of general results on the topic in the context of shifts %in the case of the shift 
we refer the reader to \cite{Cha1} and \cite{LM3} which deal with finite and compact alphabets, respectively. 
%for the case the alphabet is finite and to \cite{LM3} in the case the alphabet is a compact metric space.  
We refer the reader to \cite{Gira} for an application of Kullback-Leibler divergence in statistics. %according to (16) in \cite{Cha1} (for instance) or \cite{LM3}.
%\end{definition}

\begin{remark}
In the special case that $(\mu_\theta)_{\theta\in \Theta}$ is a parameterized family of Gibbs measures associated to normalized potentials then for $\mu_{\theta}$-almost every $x=(x_1,x_2,...,x_n,..)\in \Omega$ we have
$$
\frac{\mu_\theta\,(C_n(x))}{\mu_{\theta_0} \,( C_n(x)) }\,\sim\, e^{ -n\, h(\mu_{\theta_0} \mid \mu_\theta)}\,\,\rightarrow \,\,0,\,
$$
as $n \to \infty$, whenever $f_\theta$ and $f_{\theta_0}$ are not cohomologous. Furthermore, as the pressure function is zero in this context the relative entropy $h(\mu_{\theta_0} \mid \mu_\theta)$ can be written as
\begin{equation} \label{chazo}
h(\mu_{\theta_0} \mid \mu_\theta) = - h (\mu_{\theta_0}) -
\int \log J_\theta \, d \mu_{\theta_0}.
\end{equation}
\end{remark}

%The next proposition follows as a consequence of \cite[Section 3.2]{Cha1} (see also \cite{Cha2} or \cite{Orey}).
Expression ~\eqref{eq:computable} allows to obtain uniform estimates on the relative entropy of nearby invariant measures. More precisely:

\begin{lemma} \label{acho}
%\marginpar{\tiny I believe the strikeout text should be removed, this is not being considered at this moment}
Let $\phi : \Omega\to \mathbb R$ be a Lipschitz continuous potential and let $\mu_\phi$ be its unique Gibbs measure. Then, for any small $\varepsilon>0$ there exists $\delta>0$ such that
\begin{equation*} \label{jad}
\inf_{\mu \in\mathcal{M}_\sigma  (\Omega) }  \Big\{h(\mu \mid \mu_\phi) \,\colon\, D_\Omega (\mu, \mu_\phi) >\, \delta \,\Big\}> \varepsilon.
\end{equation*}
\end{lemma}

\begin{proof}
Fix $\varepsilon>0$. By the continuity of the map $\mu\mapsto \int \phi\, d\mu$, upper-semicontinuity of the entropy map
$\mu\mapsto h(\mu)$ and uniqueness of the equilibrium state, there exists $\delta>0$ so that any invariant probability measure $\mu$ so that $D_\Omega(\mu,\mu_\phi)>\delta$ satisfies
$
h(\mu) + \int \phi \, d \mu  < P_{\text{top}}(\sigma,\phi)  - \varepsilon.
$
This, together with \eqref{eq:computable} proves the lemma.
%$$
%h(\mu \mid \mu_\phi)
%	= P_{\text{top}}(\sigma,\phi) - \int \phi \, d \mu - h(\mu)
%	\geqslant P_{\text{top}}(\sigma,\phi) - \int \phi \, d \mu_\phi -\zeta - h(\mu)
%$$
\end{proof}

\subsection{Non-additive thermodynamic formalism}\label{subsec:almostadditive}
As mentioned before, we are mostly interested in non-additive loss functions which keep some almost additivity condition. Let us recall some of the basic notions associated to the non-additive thermodynamic formalism.

\subsubsection{Basic notions}
There are several notions of non-additive sequences which appear naturally in the description of thermodynamic objects. Let us recall some of these notions.

\begin{definition}\label{def.almost.additive}
A sequence $\Psi := \{\psi_n\}_{n\geqslant 1}$ of continuous functions $\psi_n:\Omega \to \mathbb{R}$
is called:
\begin{enumerate}
\item {\it almost additive} if there exists $C>0$ such that
\begin{equation*}
\psi_n   +\psi_m \circ \sigma^n - C \leqslant \psi_{m+n} \leqslant \psi_n + \psi_m \circ \sigma^n + C, \quad \forall m,n \geqslant 1;
\end{equation*}
\item \emph{asymptotically additive} if for any
$\xi>0$ there is a continuous function $\psi_\xi$ so that
\begin{equation*}
\limsup_{n\to\infty} \frac1n \left\| \psi_n - S_n \psi_\xi
\right\| <\xi;
\end{equation*}
\item \emph{sub-additive} if
$$
\psi_{m+n}\leqslant \psi_{m}+\psi_{n}\circ \sigma^m, \quad \forall m,n \geqslant 1.
$$
\end{enumerate}
\end{definition}

The convergence in the case of constant functions, ie sub-additive sequences is given by the following well known lemma.

\begin{lemma}[Fekete's lemma]\label{le:Fekete}
Let $(a_n)_{n\geqslant 1}$ be a sequence of real numbers so that $a_{n+m} \leqslant a_{n}+ a_{m}$ for every $n,m\geqslant 1$. Then the sequence $(a_n)_{n\geqslant 1}$ is convergent to $\inf_{n\geqslant 1} \frac{a_n}n$.
\end{lemma}

In order to recall the variational principle and equilibrium states for sequences of dynamical systems we need to obtain
an almost sure convergence. Given a probability measure $\rho \in \mathcal{M} (\Omega)$, Kingman's sub-additive ergodic theorem
ensures that any almost additive or sub-additive sequence $\Psi := \{\psi_n\}_{n\geqslant 1}$ of continuous functions is such that
\begin{equation}
 \label{rty}
 %{\mathcal F} (\rho) =
 {\mathcal F} (\rho, \Psi) = \lim_{n \to \infty} \frac{1}{n} \int \psi_n \, d \rho\,.
 \end{equation}

 \begin{definition}
 \emph{We denote by  $P_{\text{top}} (\sigma,\Phi)$ the  {\it pressure} of the almost additive family $\Phi$, associated to the family $\varphi_n$, where
$$P_{\text{top}} (\sigma,\Phi)
%	= \sup_{\rho \in \mathcal M_\sigma(\Omega)} \Big\{\, h(\rho) \,+\, \lim_{n \to \infty} \frac{1}{n} \int \varphi_n \, d \rho\,\Big\}
	=  \sup_{\rho \in \mathcal M_\sigma(\Omega)} \Big\{\, h(\rho) \,+\,{\mathcal F} (\rho, \Phi) \,\Big\}. $$
A probability measure $\mu=\mu_\Phi\in \mathcal M_\sigma(\Omega)$ is called a  {\it Gibbs measure} for the  almost additive   family $\Phi$,
if it attains the supremum.}
\end{definition}
% \begin{definition} A probability $\mu=\mu_\Phi\in \mathcal M_\sigma(\Omega)$ is called a  {\it Gibbs probability} for the  almost additive   family $\Phi$, if
% $$P_{\text{top}} (\sigma,\Phi) = \, h(\mu) \,+\, \lim_{n \to \infty} \frac{1}{n} \int \varphi_n \, d \mu\,= h(\mu) \,+\,{\mathcal F} (\mu, \Phi). $$
% \end{definition}

The previous topological pressure for non-additive sequences can also be defined, in the spirit of information theory, as the maximal topological complexity of the dynamics with respect to such sequences of observables (cf. \cite{Bar1}).
  The unique Gibbs measure associated to the family $\Phi=(\varphi_n)_{n\geqslant 1}$, $\varphi_n =\sum_{j=0}^{n-1} \log J_{\theta_0}\circ \sigma^j,$  $n \in \mathbb{N}$, is $\mu_{\theta_0}$. Moreover, in this case $P_{\text{top}} (\sigma,\Phi)=0.$ For the family $\Phi:=\{\varphi_n\}$ the claim is under the domain of the classical Thermodynamic Formalism as described before by expression (\ref{mw}).
In this case
\begin{equation} \label{mw44} P_{\text{top}} (\sigma,\Phi) = \sup_{\mu \in \mathcal M_\sigma(\Omega)} \{\, h(\mu) \,+\, \int \log J_{\theta_0} d \mu \,\}= 0. \end{equation}

\begin{remark}
In \cite{Cuneo}, the author proved that any sequence $\Psi$ of almost additive or asymptotically additive potentials is equivalent to standard additive potentials: there exists a continuous potential $\varphi$ with the same topological pressure, equilibrium states, variational principle, weak Gibbs measures, level sets (and irregular set) for the Lyapunov exponent and large deviations properties. Yet, it is still unknown wether any sequence of Lipschitz continuous potentials has a Lipschitz continuous additive representative.
\end{remark}

%\subsubsection{\color{red}  Almost-additive potentials related to inference}
\subsubsection{Almost-additive potentials related to entropy}\label{aapre}

The next proposition says that Gibbs measures determine in a natural way some sequences of almost additive potentials.

\begin{lemma}\label{le:ln-mu}
Given $\theta \in \Theta$, the family $\psi_{n,1}^\theta (y):= \log (\mu_\theta (\,C_{n} (y)\, )$, $n \in \mathbb{N},$ is almost additive.
\end{lemma}

\begin{proof}
%The argument resembles Proposition 3.1 in \cite{VZ}.
Recall that all potentials $f_\theta$ are normalized, thus each $\mu_\theta$ satisfies the Gibbs property ~\eqref{eq:Gibbs}
with $P_\theta=0$. Thus, for $\theta\in \Theta$ there exists $K_\theta>0$ such that for all $n\geqslant 1$ and $x\in \Omega$
\begin{align*}
\mu_\theta (C_{m+n}(x))
	& \leqslant K_\theta^{3} \; \mu_\theta  (C_n(x)) \,
\mu_\theta (\sigma^n  (C_{m+n}(x))\,) \\
	& = K_\theta^{3} \; \mu_\theta  (C_n(x)) \,
\mu_\theta (C_{m}(\sigma^n(x))\,).
\end{align*}
Similarly, $\mu_\theta (C_{m+n}(x)) \geqslant K_\theta^{-3} \,\mu_\theta  (C_n(x)) \, \mu_\theta (C_{m}(\sigma^n(x))\,)$ for all $n\geqslant1$.
Therefore, the family
$\psi_{n,1}^\theta ( y)= \log (\mu_\theta (\,C_n (y)\, )$ satisfies
$$ \psi_{n,1}^\theta + \psi_{m,1}^\theta \circ \sigma^n - 3 \, \log K_\theta \leqslant   \psi_{(n+m),1}^\theta \leqslant \psi_{n,1}^\theta + \psi_{m,1}^\theta \circ \sigma^n + 3 \, \log K_\theta
$$
for all $m, n\geqslant 1$, hence it is almost-additive.
\end{proof}

%\begin{prop} For all $x$ where defined $ \psi_n  (x) = \log (\,\mathbf 1_{C_n(y)}(x)\,)$ is almost additive.

%\end{prop}

%begin{proof} For all $x$ where defined $ \psi_n  (x) = \log (\,\mathbf 1_{C_n(y)}(x)\,)$ is almost additive.

%For all $x\in\Omega$ (when defined)
%$$\mathbf 1_{\overline{ y_1,y_2,...,y_{n+m} }} (x)= \mathbf 1_{C_n(y)}(x)\,\,
%\mathbf 1_{\sigma^n  (\overline{ y_1,y_2,...,y_{n+m}}) }\,(x) .$$

%\end{proof}

%\begin{prop} Given $\theta \in \Theta$, the family
%$$\varphi_{n,2}^\theta (y) = \log (\mu_\theta (\,C_n (y)\, )- \log ( \mathbf 1_{C_n (y)}\,(x)\,),$$
%$n \in \mathbb{N},$ is almost additive.

%\end{prop}

%\begin{proof} It follows from last Lemma that there exists $c_\theta$ such that
%$$ \varphi_{n,2}^\theta + \varphi_{m,2}^\theta \circ \sigma^n -\, c_\theta \leqslant   \varphi_{(n+m),2}^\theta \leqslant \varphi_{n,2}^\theta + \varphi_{m,2}^\theta \circ \sigma^n +  c_\theta,
%$$
%or all $m, n\geqslant 1$.
%\end{proof}

Note that the natural family
$$y \to -\,\log\,\int_E \,  \frac{\mu_\theta (\,C_n (y)\, )}{ \mu_{\theta_0} (\,C_n (y)\, ) } \, d \Pi_0 (\theta),$$
$n \in \mathbb{N},$ which seems at first useful, may \emph{not} be almost additive as one first evaluate fluctuations
on the different ways the measures see cylinders and only afterwards take its logarithm.
We consider alternatively the sequence of potentials given below. %by ~\eqref{kkk2} below.

\begin{lemma} \label{kkk1}
For any fixed $y\in Y$ and any Borel set $E\subset \Theta$,  the family
$$ \psi_{n}(y)= \psi_{n}^{E} (y)  = -\,\,\int  \mathbf 1_E (\theta)\, \log (\mu_\theta (\,C_n (y)\, ) \, d \Pi_0 (\theta),$$
$n \in \mathbb{N},$ is almost additive.
In particular, for each $\theta_0\in \Theta$ and $E\subset\Theta$, the family $\Psi^E := \{\Psi_n^E\}_n, $
\begin{equation}\label{kkk2}
y \to \Psi_{n}^{E} (y) := -\,\,\int_E\, \log \Big(\frac{\mu_\theta (\,C_n (y)\, )}{ \mu_{\theta_0} (\,C_n (y)\, ) }\Big) \, d \Pi_0 (\theta),
\end{equation}
is almost additive.
\end{lemma}

\begin{proof}
The first assertion is a direct consequence of the previous lemma and linearity of the integral.
For the second one just notice that %follows from the latter together with assumption \ref{372}.
$$
 -\,\,\int_E\, \log \Big(\frac{\mu_\theta (\,C_n (y)\, )}{ \mu_{\theta_0} (\,C_n (y)\, ) }\Big) \, d \Pi_0 (\theta)
 	= \psi_n(y) + \psi_{n,1}^{\theta_0}
$$
is the sum of two almost-additive sequences, hence almost additive.
\end{proof}

%\begin{prop} \label{kkk} Given a Borel set $E\subset \Theta$,  the family $\Psi^E := \{\psi_n^E\}, $
%$$y \to \psi_{n}^{E} (y) = -\,\,\int_E\, \log (\frac{\mu_\theta (\,C_n (y)\, )}{ \mu_{\theta_0} (\,C_n (y)\, ) }) \, d \Pi_0 (\theta),$$
%$n \in \mathbb{N},$ is almost additive.
%
%\end{prop}
%
%\begin{proof} It follows   at once from Proposition \ref{kkk1}.
%\end{proof}
%\medskip
%We denote by $\Psi^E$, $E \subset \Theta$, the almost-additive family $(\Psi_n^E)_{n\geqslant 1}$, where $\Psi_n^E$ is defined by \eqref{kkk2}.

\subsubsection{Almost-additive potentials related to Lyapunov exponents}\label{subsec:Lyapexp}

Let $\sigma:\Omega \to \Omega$ be a subshift of finite type and for each $\theta=(\theta_1, \theta_2,\dots,\theta_p)\in [-\varepsilon, \varepsilon]^2$ consider the locally constant linear cocycle $A_\theta: \Omega \to SL(2,\mathbb R)$
given by
$$
A_\theta\mid_{[i]}=
\begin{pmatrix}
\begin{array}{cc}
2 & 1  \\ 1 & 1
\end{array}
\end{pmatrix}
\,\cdot\,
\begin{pmatrix}
\begin{array}{cc}
\cos \theta_i & -\sin \theta_i  \\ \sin \theta_i & \cos \theta_i
\end{array}
\end{pmatrix},
\quad\text{for every $i=1, 2$. }
$$
To each $n\geqslant 1$ and $(x_1, \dots, x_n) \in \{1,2\}^n$ one associates the product matrix
$$A_\theta^{(n)}(x_1, \dots, x_n):= A_{\theta_{x_n}} \dots A_{\theta_{x_2}} A_{\theta_{x_1}}.
$$
If $\varepsilon>0$ is chosen small the previous family of matrices preserve a constant cone field in $\mathbb R^2$, hence
have a dominated splitting. Furthermore, if $\mu\in \mathcal M_\sigma(\Omega)$ is ergodic the Oseledets theorem ensures that for $\mu$-almost every $x\in \Omega$ there exists a cocycle invariant splitting
$\mathbb R^2=E^+_{\theta,x} \oplus E^-_{\theta,x}$ so that the limit
$$
\lambda_{\theta,i} := \lim_{n\to\infty} \frac1n \log \| A_\theta^{(n)}(x) \, v_i \|
$$
exists and is independent of the vector  $v_i\in E^i_{\theta,x} \setminus\{0\},$ $(i=\pm)$. Actually, Oseledets theorem also ensures that the largest Lyapunov exponent can be obtained  by means of sub-additive sequences, as
$$
\lambda^+(A_\theta,\mu)=\lambda_{\theta,+}=\inf_{n\geqslant 1} \frac1n \log \|A^{(n)}_{\theta}(x)\|,
$$
for $\mu$-almost every $x$. Since all matrices preserve a cone field then for each $\theta\in[-\varepsilon,\varepsilon]^2$ the sequence
$(\log \|A^{(n)}_{\theta}(x)\|)_{n\geqslant 1}$ is known to be almost-additive on the $x$-variable (cf.~\cite{Feng}).
Most surprisingly, in this simple context the largest annealed Lyapunov exponent
$$
\lim_{n\to\infty} \frac1n \int \log \|A^{(n)}_{\theta}(x)\| \,d\nu(x)
$$
varies analytically with the parameter $\theta$
(cf. \cite{Ruelle}).
We will need the following localization result.

\begin{lemma}\label{le:Furstenberg}
$\lambda^+(A_{(0,0)},\nu) > \lambda^+(A_\theta,\nu)$ for every $\theta\in [-\varepsilon,\varepsilon]^2 \setminus\{(0,0)\}$
\end{lemma}
\begin{proof}
First observe that, as all matrices are obtained by a rotation of the original hyperbolic matrix, we have that $\log \|A_\theta\|=\log(\frac{3+\sqrt{5}}2)$ for all $\theta\in [-\varepsilon,\varepsilon]^2$. Second,  it is clear from the definition that $\lambda^+(A_{(0,0)},\nu)$ is the logarithm of the largest eigenvalue of the unperturbed hyperbolic matrix, hence it is %$\log \|A_{(0,0)}\|
$\log(\frac{3+\sqrt{5}}2)$.
Finally, Furstenberg ~\cite{Fu} proved that
$$
\lambda^+(A_\theta,\nu) = \int  \int_{\mathbb S^1} \log \frac{\|A_\theta(x)\cdot v\|}{\|v\|} \, d\mathbb P(v)\,d\nu(x)
$$
where $\mathbb S^1$ stands for the projective space of $\mathbb R^2$ and $\mathbb P$ is a $\nu$-stationary measure,
meaning that $\nu\times \mathbb P$ is invariant by the projectivization of the map
$F(x,v)=(\sigma(x), A_{x_0}(v))$ for $(x,v)\in \Omega \times \mathbb R^2$. Altogether this guarantees that
$$
\lambda^+(A_\theta,\nu)  = \log(\frac{3+\sqrt{5}}2)
	\quad\text{if and only if}\quad
	\mathbb P=\delta_{v_+}
$$
where $v_+$ is the leading eigenvector of $A_{(0,0)}$, which cannot occur because $\nu\times \delta_{v_+}$ is not invariant by the projectivized cocycle. This proves the lemma.
\end{proof}

\medskip

\subsection{Large deviations: speed of convergence}\label{sec:speed}

Large deviations estimates are commonly used in decision theory (see e.g. \cite{Buck,FLL,Tan}). In the context of dynamical systems,
the exponential rate of convergence in large deviations are defined in terms of rate functions, often described
by thermodynamic quantities as pressure and entropy. In the
case of level-1 large deviation estimates these can be defined as follows.
%\begin{definition}
Given a  family $\Psi^E := \{\psi_n^E\},$ where $\psi_n^E:\Omega \to \mathbb{R},$ $E$ is a Borel set of parameters,  $n \in \mathbb{N}$ and 
$-\infty \leqslant c < d \leqslant \infty$, we define
$$ \overline{R}_\nu ( \Psi^E , [c,d] )= \limsup_{n \to \infty} \frac{1}{n}\,\,  \log \,\nu ( \{\,y \in \Omega : \frac{1}{n} \,\,\psi_n^E(y) \in [c,d] \}) $$
and
$$
\underline{R}_\nu ( \Psi^E , (c,d) )= \liminf_{n \to \infty} \frac{1}{n}\,\,  \log \nu ( \{\,y \in \Omega : \frac{1}{n} \,\,\psi_n^E(y) \in (c,d) \}). $$
%\end{definition}
%
Since the subshift dynamics satisfies the transitive specification property (also referred as gluing orbit property,
%an appropriate version of  \cite[Theorem B]{VZ} (see  an outline of the proof in the Appendix)
%{\bf na verdade nao se precisou tal versao geral e \cite{VZ} ja serve)}
the  \cite[Theorem B]{VZ} ensures the following large deviations principle for the subshift and either
asymptotically additive or certain sequences of sub-additive potentials.

\begin{thm}\label{thm.deviations}
Let $\Phi=\{\varphi_n\}$ be an almost additive family of potentials with $P(\sigma,\Phi)>-\infty$ and let $\nu$
be a Gibbs measure for $\sigma$ with respect to $\Phi$. Assume that either:
\begin{itemize}
\item[(a)] $\Psi=\{\psi_n\}$ is an asymptotically additive family of potentials, or;
\item[(b)] $\Psi=\{\psi_n\}$ is a sub-additive family of potentials such that:
    \begin{itemize}
    \item[i.] $\Psi=\{\psi_n\}$ satisfies the weak Bowen condition:
    		there exists $\delta>0$ so that 
$$
\limsup\limits_{n\rightarrow\infty} \frac
{\sup\{|\varphi_n(y)-\varphi_n(z)|:\
y,z \in B_n(x,\delta)\}}{n}
=0;
$$
    \item[ii.] $\inf_{n\geqslant 1} \frac{\psi_n(x)}{n}>-\infty$ for all $x\in \Omega$;  and
    \item[iii.] the sequence $\{\psi_n/n\}$ is equicontinuous.
    \end{itemize}
\end{itemize}
Given $c \in \mathbb R$, it holds that:
\begin{enumerate}
\item $\overline{R}_{\nu}(\Psi, [c,\infty))
      \leqslant \sup \big\{-P(\sigma,\Phi)+h_\eta(\sigma) + {\mathcal F}(\eta,\Phi) \big\} \}$, where the supremum is over all $\eta\in{\mathcal M}_\sigma(\Omega)$ such that ${\mathcal F}(\eta,\Psi) \geqslant c$;
\item $\underline{R}_{\nu}(\Psi, (c,\infty))
     \geqslant \sup \big\{-P(\sigma,\Phi)+h_\eta(\sigma) + {\mathcal F}(\eta,\Phi) \big\}$
where the supremum is taken over all $\eta\in{\mathcal M}_\sigma(\Omega)$ satisfying ${\mathcal F}(\eta,\Psi) > c$.
\end{enumerate}
\end{thm}

While in the previous theorem both invariant measures and sequences of observables may be generated by
non-additive sequences of potentials (we refer the reader e.g. to \cite{Bar1} for the construction of equilibrium states
associated to almost-additive sequences of potentials) we will be mostly concerned with Gibbs measures generated
by a single Lipschitz continuous potential.  In the special case of the almost-additive sequences considered in Subsection~\ref{aapre}
the previous theorem can read as follows:

\begin{corollary} \label{imp} Let $\Phi=\{\varphi_n\}$ be defined by  $\varphi_n =\sum_{j=0}^{n-1} \log J_{\theta_0},$  $n \in \mathbb{N}$ and let $\mu_{\theta_0}$ denote the corresponding Gibbs measure.  For a given Borel set  $E\subset \Theta$, take   $\Psi^E := \{\psi_n^E\},$  where $\psi_n^E$, $n \in \mathbb{N}$ was defined in Lemma~\ref{kkk1}.
Then, given $\infty \geqslant d>c\geqslant - \infty$ we have:
\begin{enumerate}
\item[a.] $\overline{R}_{\mu_{\theta_0}} ( \Psi^E , [c,d] )
	%\leqslant \sup_{\eta \in \mathcal{S}(Y)\,\,\text{such that}\,\,d > {\mathcal F} (\eta, \Psi^E) \geqslant c  }
	%\{ h(\eta) + {\mathcal F} (\eta, \Phi))\}=
	\leqslant \sup \Big\{ h(\eta) + \int \log J_{\theta_0}\, d \eta\, \colon
	{\eta \in \mathcal{S}(Y)\,\,\text{so that}\,{\mathcal F} (\eta, \Psi^E) \in [c,d]  } \Big\}$ %\leqslant 0$
\item[b.] $\overline{R}_{\mu_{\theta_0}} ( \Psi^E , (c,d) ) \geqslant
	\sup \Big\{ h(\eta) + \int \log J_{\theta_0}\, d \eta\, \colon
	{\eta \in \mathcal{S}(Y)\,\,\text{so that}\,{\mathcal F} (\eta, \Psi^E) \in (c,d)  } \Big\}$ %\leqslant 0$
\end{enumerate}
%$$\text{a)}\,\,\overline{R}_{\mu_{\theta_0}} ( \Psi^E , [c,d) ) \leqslant $$
%$$\sup_{\eta \in \mathcal{S}(Y)\,\,\text{such that}\,\,d > {\mathcal F} (\eta, \Psi^E) \geqslant c  } \{ h(\eta) + {\mathcal F} (\eta, \Phi))\}= $$
% $$\sup_{\eta \in \mathcal{S}(Y)\,\,\text{such that}\,\,d > {\mathcal F} (\eta, \Psi^E) \geqslant c  } \{ h(\eta) + \int \log J_{\theta_0}\, d \eta\,\}\leqslant 0,  $$
%
%and
%
%$$\text{b)}\,\,\underline{R}_{\mu_{\theta_0}} ( \Psi^E , (c,d) )   \geqslant \sup_{\eta \in \mathcal{S}(Y)\,\,\text{such that}\,\,d > {\mathcal F} (\eta, \Psi^E) > c  } \{ h(\eta) + \int \log J_{\theta_0}\, d \eta\,\}.  $$
%
%
\end{corollary}

As the entropy function of the subshift is upper-semicontinuous, any sequence of invariant
measures whose free energies associated to a continuous potential tend to the topological pressure accumulate
on the space of equilibrium states. Thus, in the special case that there exists a unique equilibrium state, any such sequence
is convergent to the equilibrium state. Altogether the previous argument gives the following:

\begin{lemma} \label{le:gus}
Consider the sequence of functions $\Phi=\{\varphi_n\}_{n\geqslant 1}$ where
$\varphi_n(y) = \sum_{j=0}^{n-1} \log J_{\theta_0}(\sigma^j(y))$
and $\log J_{\theta_0}$ is Lipschitz continuous, and let $\mu_\Phi$ denote the
corresponding Gibbs measure. If $U$ is an open neighborhood of the Gibbs measure $\mu_\Phi$ then there
exists $\alpha_1>0$ such that
 \begin{equation*} \label{gus}
\sup_{\mu \in \mathcal M_\sigma(\Omega)\setminus U} \{\, h(\mu) \,+\, \int \log J_{\theta_0} d \mu \}
= \sup_{\mu \in \mathcal M_\sigma(\Omega)\setminus U} \{\, h(\mu) \,+\,{\mathcal F} (\mu, \Phi) \,\}\,<\,- \alpha_1.
 \end{equation*}
\end{lemma}

\medskip

We are particularly   interested in the $\delta$-neighborhood of the parameter
$\theta_0\in \Theta$ defined by
\begin{equation}\label{def:ont0}
B_\delta = \{ \,\theta \in \Theta\,|\, d_\theta(\theta, \theta_0)<\delta\,\},\quad \text{ for some $\delta>0$.}
\end{equation}
The next result establishes large deviations estimates for relative entropy
associated to Gibbs measures close to $\mu_{\theta_0}$. More precisely:

\begin{prop} \label{rera}
Let $\Psi^E$ be defined by ~\eqref{kkk2}.
For any  $\delta>0$ there exists $d_\delta>0$ satisfying
 $$ \,{\mathcal F} (\mu_{\theta_0}, \Psi^{B_\delta}) \,\,  < d_\delta< {\mathcal F} (\mu_{\theta_0}, \Psi^\Theta)=  \int_\Theta h(\mu_{\theta_0} \mid \mu_\theta )  \, d \Pi_0 (\theta).$$
 Moreover, %there exists $\delta_0$, such that, for any $0<\delta<\delta_0$ one can find $\alpha_1>0$ such that
 for every small $\delta>0$ there exists $\alpha_1>0$ so that
\begin{equation*} \label{cgl}
\limsup_{n \to \infty} \frac{1}{n}\,\,  \log \Big[\,\mu_{\theta_0} \Big( \{\, y\in \Omega \,\colon \, - \frac{1}{n}
\int_{B_\delta} \log \Big(\frac{\mu_\theta (\,C_n (y)\, )}{ \mu_{\theta_0} (\,C_n (y)\, ) }\Big)\, \, d \Pi_0 (\theta)   \in [d_\delta, \infty)    \}  \Big)  \,\Big]
	\leqslant - \alpha_1.
\end{equation*}
\end{prop}

\medskip
\begin{proof}
%We set $\gamma =   \int_{\Theta} \int_\Omega \log J_\theta \, d  \mu_{\theta_0}  \,d \Pi_0.$
Remember %from Definition \ref{rty}
that, given $\eta \in \mathcal{M}_\sigma( \Omega)$ and $E\subset \Theta$,
$$-\lim_{n \to \infty} \frac{1}{n}
\int\,\int_E \,\log
\Big(\frac{\mu_\theta (\,C_n (y)\, )}{ \mu_{\theta_0} (\,C_n (y)\, ) }\Big)\,
d \Pi_0 (\theta)\, d \eta (y)
=
\lim_{n \to \infty} \frac{1}{n} \int \psi_n^E (y) \, d \eta(y)\,=  \,{\mathcal F} (\eta, \Psi^E).
$$
Taking $\eta = \mu_{\theta_0}$ and $E=\Theta$ we get from (\ref{Chacha}), (\ref{chazo}) and Lemma~\ref{acho}
that
\begin{align}
{\mathcal F} (\mu_{\theta_0}, \Psi^\Theta)
%	& =-\lim_{n \to \infty} \frac{1}{n}
%\int\,[\,\int_\Theta \log (\frac{\mu_\theta (\,C_n (y)\, )}{ \mu_{\theta_0} (\,C_n (y)\, ) })\, d \Pi_0 (\theta)\,]\, d \mu_{\theta_0} (y) \\
	& = \int_\Theta
	\int - \lim_{n \to \infty} \frac{1}{n} \log
	\Big(\frac{\mu_\theta (\,C_n (y)\, )}{ \mu_{\theta_0} (\,C_n (y)\, ) }\Big)\,  \, d \mu_{\theta_0} (y)\,
	\,d \Pi_0 (\theta) \nonumber \\
	& = \int_\Theta h(\mu_{\theta_0} \mid \mu_\theta )  \,d \Pi_0 (\theta)\nonumber  \\
	& = \label{sky}   - h (\mu_{\theta_0}) -
\int_\Theta \int \log J_\theta \,d \mu_\theta \,d \Pi_0. %=    - h (\mu_{\theta_0}) - \gamma >0 .
\end{align}
Similarly, one obtains
$
{\mathcal F} (\mu_{\theta_0}, \Psi^E)  =    - h (\mu_{\theta_0}) \, \Pi_0(E)-
\int_{E}  \int \log J_\theta \,d \mu_\theta \,d \Pi_0
$
  for any $E \subset \Theta$.
Using that $h(\mu_{\theta_0} \mid \mu_\theta )>0$ for all $\theta\neq\theta_0$ and that $\Pi_0$ is fully supported on $\Theta$,
%
%From (\ref{sky}) we get that
%\marginpar{\tiny \color{red} why?}
%$$
%{\mathcal F} (\mu_{\theta_0}, \Psi^\Theta) = \int_\Theta h(\mu_{\theta_0} \mid \mu_\theta )  \, d \Pi_0 (\theta)= - \gamma - h(\mu_{\theta_0})>0.
%$$
%By assumption, the a priori measure $\Pi_0$ is full supported on $\Theta = [a_1,b_1] \times [a_2,b_2] \times ...\times [a_k,b_k]$.
Lemma~\ref{acho} ensures that
$$
%\int_{ \{\theta\,| \,D_\Omega (\mu_{\theta_0}, \mu_\theta) \geqslant \, \delta \,\}  }h(\mu_{\theta_0} \mid \mu_\theta )
	%\, d \Pi_0 (\theta)
	 \int_{\Theta \setminus B_\delta}h(\mu_{\theta_0} \mid \mu_\theta )  \,d \Pi_0 (\theta) > 0
$$
for every small $\delta$.
%From expression (\ref{jad}) in Proposition \ref{acho}  given  $\delta>0$ there exists $%\varepsilon$, such that,
%Together with ~\eqref{sky}, the previous argument implies that
In consequence,
\begin{align*}
{\mathcal F} (\mu_{\theta_0}, \Psi^{B_\delta}) =
%= - \lim_{n \to \infty}  \frac{1}{n}  \int_{B_\delta } \,[\int \log (\frac{\mu_\theta (\,C_n (y)\, )}{ \mu_{\theta_0} (\,C_n (y)\, ) }) \, d \mu_{\theta_0} (y)\,]\,d \Pi_0 (\theta)=$$
%
%\int_{B_\delta}\,\int -\lim_{n \to \infty} \frac{1}{n}\log (\frac{\mu_\theta (\,C_n (y)\, )}{ \mu_{\theta_0} (\,C_n (y)\, ) })\, d \Pi_0 (\theta)\, d \mu_{\theta_0} (y)=$$
\int_{B_\delta}\, h(\mu_{\theta_0},\mu_\theta)\, d \Pi_0 (\theta)\,
	& <  \int_{B_\delta} \,\int - \lim_{n \to \infty} \frac{1}{n} \log (\frac{\mu_\theta (\,C_n (y)\, )}{ \mu_{\theta_0} (\,C_n (y)\, ) }) \, d \mu_{\theta_0} (y)\,\,d \Pi_0 (\theta)\, \\
	& +
\int_{\Theta \setminus B_\delta} \,\int - \lim_{n \to \infty} \frac{1}{n} \log (\frac{\mu_\theta (\,C_n (y)\, )}{ \mu_{\theta_0} (\,C_n (y)\, ) }) \, d \mu_{\theta_0} (y)\,\,d \Pi_0 (\theta) \\
	&  = \int_\Theta h(\mu_{\theta_0} \mid \mu_\theta )  d \Pi_0 (\theta)
	= {\mathcal F} (\mu_{\theta_0}, \Psi^{\Theta}).
\end{align*}
for every small $\delta$, hence there exists $d_\delta>0$ so that
\begin{equation} \label{gtk}     {\mathcal F} (\mu_{\theta_0}, \Psi^{B_\delta}) < d_\delta<{\mathcal F} (\mu_{\theta_0}, \Psi^{\Theta}). \end{equation}
%
%Note that when $\delta \to 0$, we get
%$${\mathcal F} (\mu_{\theta_0}, \Psi^{B_\delta}) \to  \int_\Theta h(\mu_{\theta_0}, %%\mu_\theta) \, d \Pi_0 (\theta) = {\mathcal F} (\mu_{\theta_0}, \Psi^{\Theta})= %h(\mu_{\theta_0}, \mu_\theta) =0.$$
%
Now, on the one hand, by continuity of $\eta\mapsto {\mathcal F} (\eta, \Psi^{B_\delta})$, %For $\delta$ fixed but small enough
the set
$U=\{\,\eta \in \mathcal M_\sigma(\Omega)\,\,:\,\,{\mathcal F} (\eta, \Psi^{B_\delta})< d _\delta  \}$
is an open neighborhood of $\mu_{\theta_0}.$ On the other hand, according to Lemma~\ref{le:gus} there exists $\alpha_1>0$ such that
$$
\sup_{\eta \in \mathcal M_\sigma(\Omega) \setminus U } \Big\{ h(\eta) + \int \log J_{\theta_0}\, d \eta\,\Big\}
	\leqslant - \alpha_1<0.
$$
Therefore, from~Theorem \ref{imp}
\begin{align*}
\limsup_{n \to \infty} \frac{1}{n}\,\,  & \log \mu_{\theta_0} \Big( \{\, y \in \Omega\,|\, - \frac{1}{n}
\int_{E^\delta}\log (\frac{\mu_\theta (\,C_n (y)\, )}{ \mu_{\theta_0} (\,C_n (y)\, ) })   \, d \Pi_0 (\theta)\,        \in [d_\delta, \infty)      \}  \Big)   \\
& \leqslant
 \sup_{\{\eta \in \mathcal M_\sigma(\Omega)\,\colon \,{\mathcal F} (\eta, \Psi^{B_\delta}) > d_\delta \}} \{ h(\eta) + \int \log J_{\theta_0}\, d \eta\,\} \leqslant \,-\,\alpha_1<0.
\end{align*}

%Note that if $\delta$ is very large, then $B_\delta$ can be very large and so in this case the value ${\mathcal F} (\mu_{\theta_0}, \Psi^{B_\delta})$ is not very well controlled.

\end{proof}

 %%$$ ?? {\mathcal F} (\mu_{\theta_0}, \Psi^{\Theta}) = h ( \mu_{\theta_0} )+\, \int_\Theta h(\mu_{\theta_0}, \mu_\theta) \, d \Pi_0 (\theta).$$

 \medskip

 \begin{rmk} \label{rema}
 From Hypothesis A the value $d_\delta>0$ can be taken small, if $\delta>0$ is small,  because
 $ \,{\mathcal F} (\mu_{\theta_0}, \Psi^{B_\delta}) \,=\,  \int_{B_\delta}\, h(\mu_{\theta_0}\mid\mu_\theta)\, d \Pi_0 (\theta)\,   .
 $
 \end{rmk}

 \medskip

\begin{cor} \label{rer3} Given $\delta>0$ small let $B_\delta\subset \Theta$ be the $\delta$-open neighborhood of $\theta_0$
defined in \eqref{def:ont0} and let $d_\delta>0$ be given by Proposition \ref{rera}.
%There exists $\delta_0>0$ so that for each $0<\delta<\delta_0$ there exists $d_\delta>0$ satisfying
The following holds:
\begin{equation} \label{ss}  \limsup_{n \to \infty}  - \frac{1}{n}\log
\int_{B_\delta} \frac{ \mu_{\theta}  (\,C_n (y)\, ) }{ \mu_{\theta_0}  (\,C_n (y)\, )}\, d \Pi_0 (\theta)     \leqslant d_\delta
\end{equation}
for  $\mu_{\theta_0}$-almost every point $y$.
Moreover, for  $\mu_{\theta_0}$-almost every point $y$
\begin{equation} \label{sis}   \liminf_{n \to \infty} \frac{1}{n} \log\,\,
\int_{B_\delta}   \mu_{\theta}  (\,C_n (y)\, )\, d \Pi_0 (\theta) \geqslant - \Pi_0 (B_\delta)\cdot h(\mu_{\theta_0})- d_\delta.\end{equation}
\end{cor}

\begin{proof} For each $n\geqslant 1$ consider the set
$$A_n= \Big\{\, y\in \Omega\,|\, - \frac{1}{n}
\int_{B_\delta}  \log  \frac{ \mu_\theta  (\,C_n (y)\, ) }{ \mu_{\theta_0}  (\,C_n (y)\, )} \, d \Pi_0 (\theta)   \in [d_{\delta},  \infty)      \Big\} . $$
%From (\ref{cgl}) in
By Proposition \ref{rera}, we get that $\sum_n \mu_{\theta_0} (A_n) <\infty.$ It follows from   Borel-Cantelli Lemma that for $\mu_{\theta_0}$-almost every point $y\in\Omega$ there exists an $N$, such that
$y\notin A_n$ for all $n>N$. Equivalently,  $- \frac{1}{n}
\int_{B_\delta} \log(  \frac{ \mu_{\theta}  (\,C_n (y)\, ) }{ \mu_{\theta_0}  (\,C_n (y)\, )}\,)\, d \Pi_0 (\theta)< d_\delta$
for all $n>N$, which proves \eqref{ss}.
Therefore, from Jensen inequality, we get for $\mu_{\theta_0}$-almost every $y\in\Omega$ and every large $n\geqslant 1$
 \begin{align}
\nonumber    \frac{1}{n} & \log\,\,
\int_{B_\delta} \,  \mu_{\theta}\  (\,C_n (y)\, )\, d \Pi_0 (\theta)\,\,  - \frac{1}{n} \int_{B_\delta}  \log ( \mu_{\theta_0}  (\,C_n (y)\, )\,)\, d \Pi_0 (\theta)   \\
\nonumber	& \geqslant  \frac{1}{n} \int_{B_\delta}  \log(\,   \mu_{\theta}  (\,C_n (y)\, ) \, d \Pi_0 (\theta) - \frac{1}{n} \int_{B_\delta}  \log ( \mu_{\theta_0}  (\,C_n (y)\, )\,)\, d \Pi_0 (\theta) \\
	& \label{tresu} = \frac{1}{n} \int_{B_\delta}  \log  \frac{ \mu_{\theta}  (\,C_n (y)\, ) }{ \mu_{\theta_0}  (\,C_n (y)\, )}\, d \Pi_0 (\theta) \geqslant -  d_\delta.
\end{align}
Moreover,  as $\lim_{n\to\infty} - \frac{1}{n} \log ( \mu_{\theta_0}  (\,C_n (y)\, )=h(\mu_{\theta_0})$ for $\mu_{\theta_0}$-almost every $y$, it follows from the previous inequalities that
$$
 \liminf_{n \to \infty} \frac{1}{n} \log\,\,
\int_{B_\delta} \,  \mu_{\theta}  (C_n (y))\, d \Pi_0 (\theta)\,\, +
\Pi_0 (B_\delta)\, h(\mu_{\theta_0}) \geqslant - d_\delta
$$
for $\mu_{\theta_0}$ almost every $y$, which proves ~\eqref{sis}, as desired.
%Therefore, for $\mu_{\theta_0}$-almost every $y$, if  $n$ is large enough then
%$$ \int_{B_\delta} \,  \mu_{\theta}  (\,C_n (y)\, )\, d \Pi_0 (\theta) \geqslant         e^{- (d_\delta + \Pi_0 (B_\delta)\,\, h(\mu_{\theta_0})\,) \,n}.$$
\end{proof}

\smallskip

\begin{rmk} \label{wer}
\emph{
The previous corollary ensures that for any $\zeta>0$ and $\mu_{\theta_0}$-a.e. $y\in \Omega$
%there exists a subsequence $(n_k)_{k\geqslant 1}$tending to infinity so that
%\begin{equation*} \label{sus}
% \int_{B_\delta} \,  \mu_{\theta}  (\,C_{n_k} (y)\, )\, d \Pi_0 (\theta) \geqslant   e^{- [d_\delta + \Pi_0 (B_\delta)\,\, h(\mu_{\theta_0})-\zeta] \,{n_k}} \quad \text{for every large $k\geqslant 1$.}
%\end{equation*}
\begin{equation*} \label{sus}
 \int_{B_\delta} \,  \mu_{\theta}  (\,C_{n} (y)\, )\, d \Pi_0 (\theta) \geqslant   e^{- [d_\delta + \Pi_0 (B_\delta)\,\, h(\mu_{\theta_0})+\zeta] \,{n}} \quad \text{for every large $n\geqslant 1$.}
\end{equation*}
Moreover, Remark \ref{rema} guarantees that $d_\delta>0$ can be chosen small provided that $\delta$ is small.
In particular, the absolute continuity assumption on the \emph{a priori} measure $\Pi_0$ (hypothesis A) implies
that $\Pi_0 (B_\delta)\, h(\mu_{\theta_0})+ d_\delta$ can be taken arbitrarily small, provided that $\delta$ is small.
}
\end{rmk}

\color{black}

\smallskip

%The next result corresponds to Theorem 2 in \cite{Nobel1} for the particular case of direct observation.

\medskip

\begin{lemma} \label{ufa} For small $\delta>0$ and $\mu_{\theta_0}$-almost every $y\in \Omega$
\begin{equation} \label{lok}
 \limsup_{n \to \infty}  \frac{1}{n}\,\log  \int_{\Theta \setminus B_\delta}     \mu_\theta (C_n (y) ) \,d\Pi_0 (\theta)
	\leqslant \sup_{\theta \in \Theta \setminus B_\delta} \int \log J_\theta d \mu_{\theta_0}<0.
\end{equation}
Moreover, $\sup_{\theta \in \Theta \setminus B_\delta} \int \log J_\theta \,d \mu_{\theta_0}\to - h(\mu_{\theta_0})$
as $\delta\to 0$.

\end{lemma}

\begin{proof}
Recalling the Gibbs property ~\eqref{eq:Gibbs} for $\mu_\theta$ the continuous dependence of the constants $K_\theta$
and compactness of $\Theta$ we conclude that there exist uniform constants $c_1, c_2>0$ so that
\begin{equation}\label{eq:Gibbs2}
{c_1} \leqslant \frac{\mu_\theta(C_n(x))}{e^{-n P_\theta} + S_nf_\theta(x)} \leqslant {c_2}
	\qquad \forall \theta \in \Theta, \forall y\in \Omega, \forall n\geqslant 1.
\end{equation}
%There exist uniform constants $c_1>0$ and $c_2>0$, such that, for all $\mu_\theta$, $\theta\in \Theta$, and $y$ and all $n$
%
%$$c_1 <\frac{ \mu_\theta (C_n (y) ) }{e^{ \sum_{j=0}^{n-1} \log J_\theta (\sigma^j( y)) } }< c_2$$
%
Furthermore, as the potentials are assumed to be normalized then $P_\theta=0$ for every $\theta\in \Theta$.
Therefore, there exists $C_1>0$ and $C_2>0$ , such that, for all $y\in \Omega$, $\theta\in \Theta$ and $n\geqslant 1$
$$C_1< \frac{ \mu_\theta (C_n (y) ) }{ \mu_{\theta_0} (C_n (y) ) } \frac{ e^{ \sum_{j=0}^{n-1} \log J_{\theta_0} (\sigma^j( y))   }}{e^{ \sum_{j=0}^{n-1} \log J_\theta (\sigma^j( y))}  }  <C_2.$$
Then,
\begin{align*}
\limsup_{n \to \infty} \frac{1}{n} \log C_1
	& <  \limsup_{n \to \infty} \frac{1}{n}  \log  \frac{ \mu_\theta (C_n (y) ) }{ \mu_{\theta_0} (C_n (y) ) }   \\
	& + \limsup_{n \to \infty} \frac{1}{n}    \sum_{j=0}^{n-1} \log J_{\theta_0} (\sigma^j( y))   - \sum_{j=0}^{n-1} \log J_\theta (\sigma^j( y))  ] \\
	& < \limsup_{n \to \infty} \frac{1}{n} \log C_2.
\end{align*}
In consequence, using the ergodic theorem and that $h(\mu_{\theta_0} )=\int -\log J_{\theta_0}\, d\mu_{\theta_0}$, one gets
$$ \limsup_{n \to \infty} \frac{1}{n} \log  \frac{ \mu_\theta (C_n (y) ) }{ \mu_{\theta_0} (C_n (y) ) }
	\leqslant h(\mu_{\theta_0} ) + \int \log J_\theta  d \mu_{\theta_0},
	\quad
	\text{for $\mu_{\theta_0}$-a.e. $y$.}
$$
Fix $\zeta>0$ arbitrary and small. The previous expression ensures that, for $\mu_{\theta_0}$-a.e. $y\in \Omega$,
$$  \frac{ \mu_\theta (C_n (y) ) }{ \mu_{\theta_0} (C_n (y) ) }  \leqslant e^{n\, (h(\mu_{\theta_0} ) + \int \log J_\theta  d \mu_{\theta_0}  +\zeta )}
\quad\text{for every large $n\geqslant 1$}
$$
Given a small $\delta>0$, by uniqueness of the equilibrium state for $\log J_\theta$, we have that
$$
\rho_\delta:=h(\mu_{\theta_0} ) + \sup_{\theta \in \Theta \setminus B_\delta}  \int \log J_\theta d \mu_{\theta_0}=
\sup_{\theta \in \Theta \setminus B_\delta} [h(\mu_{\theta_0} ) + \int \log J_\theta d \mu_{\theta_0}] <0,
$$
and that $\rho_\delta$ tends to zero as $\delta\to 0$.
Then, for $\mu_{\theta_0}$-almost every point $y$
$$
 \int_{\Theta \setminus B_\delta}      \frac{ \mu_\theta (C_n (y) ) }{ \mu_{\theta_0} (C_n (y) ) } d \Pi_0 (\theta)
 	\leqslant   \int_{\Theta \setminus B_\delta} e^{n\, (h(\mu_{\theta_0} ) +   \log J_\theta  d \mu_{\theta_0} +\zeta)}
		d \Pi_0 (\theta)
	\leqslant \Pi_0 (\Theta \setminus B_\delta) e^{n (\rho_\delta-\zeta)}
$$
which implies
$$
 \limsup_{n \to \infty}  \frac{1}{n}\,\log  \int_{\Theta \setminus B_\delta}     \frac{ \mu_\theta (C_n (y) ) }{ \mu_{\theta_0} (C_n (y) ) } d \Pi_0 (\theta) \leqslant   \rho_\delta-\zeta.
$$
As $\zeta>0$ was chosen arbitrary we conclude that, for $\mu_{\theta_0}$-almost every point $y$,
$$ \limsup_{n \to \infty}  \frac{1}{n}\,\log  \int_{\Theta \setminus B_\delta}     \mu_\theta (C_n (y) ) d \Pi_0 (\theta)
	\leqslant  - h(\mu_{\theta_0}) + \rho_\delta =
   \sup_{\theta \in \Theta \setminus B_\delta} \int \log J_\theta d \mu_{\theta_0}<0.$$
\end{proof}

%\color{gray}

\begin{prop} For $\mu_{\theta_0}$-almost every $y\in \Omega$
\begin{equation*} \label{tri9}
0\leqslant -\limsup_{n \to \infty} \frac{1}{n} \log Z_n(y) \,\,
	\leqslant     -  \int_{\Theta} \int_\Omega \log J_\theta(y) d  \mu_{\theta_0}(y)d \Pi_0 (\theta) - h( \mu_{\theta_0}).
\end{equation*}
\end{prop}

%\begin{proof}

% For each $\delta>0$, it follows from Lemma \ref{ufa} that for $y$ almost everywhere point with respect to $\mu_{\theta_0}$.

%\begin{align*}
%-\sup_{\theta \in \Theta \setminus B_\delta} \int \log J_\theta d \mu_{\theta_0}
%	& \leqslant - \limsup_{n \to \infty} [\frac{1}{n} \log \,[\,\int_{\Theta \setminus B_\delta}  \mu_{\theta}  \,(\,C_n(y) \,)   d \Pi_0 (\theta) \\
%	& = - \limsup_{n \to \infty} [\frac{1}{n} \log \,[\,\int_{\Theta \setminus %B_\delta}  \mu_{\theta}  \,(\,C_n(y) \,)   d \Pi_0 (\theta)	 +  \int_{B_\delta} \mu_{\theta}  (C_n(y) )        d \Pi_0 (\theta) ]   \\
%	& =  - \limsup_{n \to \infty} \frac{1}{n} \log\int_{\Theta} \mu_{\theta}  (C_n(y) )        d \Pi_0 (\theta) ].
%\end{align*}

%From Lemma \ref{ufa} we have that $-\sup_{\theta \in \Theta \setminus B_\delta} \int \log J_\theta d \mu_{\theta_0} \to  h( \mu_{\theta_0})$, when $\delta \to 0$.
%\end{proof}

%\color{black}

The statement of the second inequality of this Proposition  is nothing more than the expression (\ref{conseq:vp}).

\smallskip

%%%%%%%%%%%%%%%%%%%%%%%%
\section{Proof of the main results}\label{sec:proofs}

%%%%%%%%%%%%%%%%%%%%%%%%
\subsection{Proof of Theorem~\ref{thm:main}}

We proceed to show that the \emph{a posteriori} measures in Theorem~\ref{thm:main} do converge, for  $\mu_{\theta_0}$-typical points $y$. In order to prove that $\Pi_n(\cdot,y) \to \delta_{\theta_0}$ (in the weak$^*$ topology) it is sufficient
to prove that, for every $\delta>0$, one has that  $\Pi_n(\Theta\setminus B_\delta,y) \to 0$ as $n\to\infty$. This is the content of the following theorem.
\medskip

\begin{thm}  \label{mmai}
Let $\Pi_n (\cdot \mid y)$ be the a posteriori measures defined by ~\eqref{frt} and let
$B_\delta$ be the $\delta$-neighborhood of $\theta_0$ defined by ~\eqref{def:ont0}.
Then,  for every small $\delta>0$ and  $\mu_{\theta_0}$-a.e. $y$,
\begin{equation} \label{espe}
\Pi_n (B_\delta \mid y)=\frac{\int_{B_\delta} \mu_{\theta}  \,(\,C_n (y))\, d \Pi_0 (\theta) \,}{\int_\Theta \mu_{\theta}(C_n (y)\,)d \Pi_0 (\theta) \, } \to 1
\end{equation}
exponentially fast as $n \to \infty.$

%Then, for $\mu_{\theta_0}$ almost very point $y$, there exist a positive $u$, such that,

%\begin{equation} \label{yrt1}  \limsup_{n \to \infty} \frac{1}{n} \log \Pi_n (B_\delta, y, n) > u>0 .
%\end{equation}
%\end{thm}
\end{thm}

\begin{proof}
Fix $\delta>0$ small. We claim that $\Pi_n (\Theta \setminus B_\delta \mid y) $ tends to zero exponentially fast
as $n \to \infty$.
We have to estimate
$$ \limsup_{n \to \infty} \frac{1}{n} \log \,\int_{B_\delta} \mu_\theta (C_n (y) ) d \Pi_0 (\theta)$$
and
$$- \limsup_{n \to \infty} \frac{1}{n} \log \int_{\Theta \setminus B_\theta} \mu_\theta (C_n (y) ) d \Pi_0 (\theta).$$
From (\ref{sis}), for $\mu_{\theta_0}$ almost every point $y$
\begin{equation} \label{soso0}   \limsup_{n \to \infty} \frac{1}{n} \log\,\,
\int_{B_\delta} \,  \mu_{\theta}  (C_n (y))\, d \Pi_0 (\theta)\,\, \geqslant - h(\mu_{\theta_0})\,  \Pi_0 (B_\delta)- d_\delta,
\end{equation}
where $d_\delta$ can be taken small if $\delta>0$ is small. Fix $0<\zeta<\frac{h(\mu_{\theta_0})}2$.
Therefore, from Remark~\ref{wer} %(\ref{sus})
we get that for $\mu_{\theta_0}$ almost every point $y$
\begin{equation} \label{suspi}
 \int_{B_\delta} \,  \mu_{\theta}  (\,C_{n} (y)\, )\, d \Pi_0 (\theta) \geqslant   e^{- [d_\delta + \Pi_0 (B_\delta)\,\, h(\mu_{\theta_0})-\zeta] \,{n}} \quad \text{for every large $n\geqslant 1$.}
\end{equation}
%\begin{equation} \label{suspi}  \int_{B_\delta}   \mu_{\theta}  (\,C_n (y)\,)\,    d \Pi_0 (\theta) \geqslant e^{ - (h(\mu_{\theta_0})\,  \Pi_0 (B_\delta)+ d_\delta-\zeta)\, \,n}.
%\end{equation}
%\smallskip
%
Observe that the map
$\delta\mapsto \sup_{\theta \in \Theta \setminus B_\delta} \int \log J_\theta d \mu_{\theta_0}$ is monotone increasing and recall that $\sup_{\theta \in \Theta \setminus B_\delta} \int \log J_\theta d \mu_{\theta_0} \to - h(\mu_{\theta_0})$
as $\delta\to 0$. On the other hand, $- h(\mu_{\theta_0})\,  \Pi_0 (B_\delta)- d_\delta$ tends to zero as $\delta \to 0$
(cf. Remark~\ref{wer}). Thus,
\begin{equation} \label{now}
   \sup_{\theta \in \Theta \setminus B_\delta} \int \log J_\theta d \mu_{\theta_0}	
   	< - h(\mu_{\theta_0})\,  \Pi_0 (B_\delta)- d_\delta -\zeta
\end{equation}
for every small $\delta>0$.
As $$ \frac{\int_{B_\delta} \mu_{\theta}  \,(\,C_n (y))\, d \Pi_0 (\theta) \,}{\int_\Theta \mu_{\theta}(C_n (y)\,)d \Pi_0 (\theta) \, }   +
 \frac{\int_{\Theta \setminus B_\delta} \mu_{\theta}  \,(\,C_n (y)\,d \Pi_0 (\theta)  \,}{\int_ \Theta \mu_{\theta}(C_n (y)\,d \Pi_0 (\theta) \, }   =1$$
 we just have to show that
 $$\frac{\int_{B_\delta} \mu_{\theta}  \,(\,C_n (y))\, d \Pi_0 (\theta) \,}{\int_{\Theta \setminus B_\delta} \mu_{\theta}(C_n (y)\,)d \Pi_0 (\theta) \, } \to \infty,  $$
 when $n \to \infty.$
Now, equations \eqref{lok} and \eqref{suspi} and the choice of $\delta$ in (\ref{now}) ensure that, for $\mu_{\theta_0}$-almost every $y\in \Omega$,
 $$
 \frac{\int_{B_\delta} \mu_{\theta}  \,(\,C_n (y))\, d \Pi_0 (\theta) \,}{\int_{\Theta \setminus B_\delta} \mu_{\theta}(C_n (y)\,)d \Pi_0 (\theta) \, }
 	\geqslant \frac{ e^{ - [h(\mu_{\theta_0})\,  \Pi_0 (B_\delta)+ d_\delta+\zeta]\, \,n}}{
		 e^{ n\, \sup_{\theta \in \Theta \setminus B_\delta} \int \log J_\theta d \mu_{\theta_0}} }
$$
 which tends to infinity as $n\to\infty$. Finally the previous expression also ensures that
 \begin{equation*}
| \Pi_n (B_\delta \mid y)-1 |
	=\frac{\int_{\Theta\setminus B_\delta} \mu_{\theta}  \,(\,C_n (y))\, d \Pi_0 (\theta) \,}{\int_\Theta \mu_{\theta}(C_n (y)\,)d \Pi_0 (\theta) \, }
	\leqslant
		 e^{ n\,[ \sup_{\theta \in \Theta \setminus B_\delta} \int \log J_\theta d \mu_{\theta_0} + h(\mu_{\theta_0})\,  \Pi_0 (B_\delta)+ d_\delta+\zeta] }
\end{equation*}
decreases exponentially fast with exponential rate that can be taken uniform for all small $\delta>0$.
  This finishes the proof of the theorem.
 \end{proof}

%As $\int_\Theta \mu_\theta (C_n (y) ) d \Pi_0 (\theta)\leqslant 1,$
%then, for $\mu_{\theta_0}$ almost every point $y$
%$$\Pi_n (B_{\delta}, y) = \, \frac{ \int_{B_\delta} \mu_\theta (C_n (y) ) d \Pi_0 (\theta) }{ \int_\Theta \mu_%\theta (C_n (y) ) d \Pi_0 (\theta)}\geqslant e^{ - (h(\mu_{\theta_0})\,  \Pi_0 (B_\delta)+ d_\delta)\, \,n}. $$

%%%%%%%%%%%%%%%%%%%%%%%%
\subsection{Proof of Theorem~\ref{thm:main2}}

By assumption, there exists a full $\nu$-measure subset $Y'\subset Y$ so that
the limit
	$$
	\Gamma^y(\theta):=\lim_{n\to\infty} \frac1n \log \int_\Omega e^{\varphi_n(\theta, x,y)} \, d\mu_\theta(x)
	$$
exists for every $y\in Y'$. Given an arbitrary $y\in Y'$ we proceed to estimate the asymptotic behavior of
the \emph{a posteriori} measures $\Pi_n (\cdot \mid y)$ given by ~\eqref{frtnew2}.

Given $\delta>0$, by upper semicontinuity of $\Gamma^y(\cdot)$ the function $\Gamma^y$ has a maximum value
and there exists $d_\delta>0$ (which may be chosen to converge to zero as $\delta\to 0$) so that
$$
B_\delta^y=\big\{\theta \in \Theta \colon d_\Theta\big(\,\theta, \text{argmax} \, \Gamma^y\,\big)> \delta \big\}
	\subset (\Gamma^y)^{-1}(( -\infty, \alpha^y -d_\delta))
$$
is non-empty and open subset, where $\alpha^y:=\max_{\theta\in \Theta} \Gamma^y(\theta)$.

There are two cases to consider. On the one hand, if $\Gamma^y(\cdot)\equiv \alpha^y$ is constant
then $B_\delta^y=\Theta $ and we conclude that $\Pi_n (B_\delta^y \mid y)  = 1$ for all $n\geqslant 1$
and the convergence in ~\eqref{eq:thm:main2} is trivially satisfied.
On the other hand, as $\Pi_0$ is fully supported and absolutely continuous then $\int_\Theta \Gamma^y(\theta) \, d\Pi_0(\theta)<\alpha^y$.
Actually, this allows to estimate the double integral
$$
\int_{\Theta\setminus B^y_\delta} \int_\Omega e^{\varphi_n(\theta, x,y)} \, d\mu_\theta(x) \, d\Pi_0(\theta)
$$
without making use of the features of the set $B^y_\delta$. More precisely, using Jensen inequality and taking the limsup under the sign of the integral,
\begin{align*}
\limsup_{n\to\infty} \frac1n \log \int_{\Theta\setminus B^y_\delta} \int_\Omega e^{\varphi_n(\theta, x,y)} \, d\mu_\theta(x) \, d\Pi_0(\theta)
	& \leqslant  \int_{\Theta\setminus B^y_\delta} \limsup_{n\to\infty} \frac1n \log \int_\Omega e^{\varphi_n(\theta, x,y)} \, d\mu_\theta(x) \, d\Pi_0(\theta) \\
	& =  \int_{\Theta\setminus B^y_\delta} \Gamma^y(\theta)  \, d\Pi_0(\theta).
\end{align*}
As $\varphi_n$ are assumed non-negative we conclude that $\Gamma^y(\cdot)$ is a non-negative function and
\begin{align}\label{eq:keyl}
\limsup_{n\to\infty} \frac1n \log \int_{\Theta} \int_\Omega e^{\varphi_n(\theta, x,y)} \, d\mu_\theta(x) \, d\Pi_0(\theta)
	& \leqslant \int_{\Theta} \Gamma^y(\theta)  \, d\Pi_0(\theta)<\alpha^y.
\end{align}
In consequence, if $0<\zeta<\frac12\big[\alpha^y-\int_{\Theta} \Gamma^y(\theta)  \, d\Pi_0(\theta)\big]$ then
$$
\int_{\Theta} \int_\Omega e^{\varphi_n(\theta, x,y)} \, d\mu_\theta(x) \, d\Pi_0(\theta) \leqslant e^{(\alpha^y-\zeta)n}
$$
for every large $n\geqslant 1$. Now, in order to estimate the measures $\Pi_n (\cdot \mid y)$ on the nested family $(B_\delta^y)_{\delta>0}$ we observe that
$
\int_\Omega e^{\varphi_n(\theta, x,y)} \, d\mu_\theta(x) \geqslant e^{(\alpha^y-d_\delta)n},
	\; \forall
	\theta \in B_\delta^y,
$
thus
$$
\int_{B^y_\delta} \int_\Omega e^{\varphi_n(\theta, x,y)} \, d\mu_\theta(x) \, d\Pi_0(\theta)
	\geqslant e^{(\alpha^y-d_\delta)n} \Pi_0(B^y_\delta)
$$
for every large $n\geqslant 1$. In particular, if $\delta>0$ is small so that $0<d_\delta<\zeta$, putting together the last expression, inequality \eqref{eq:keyl} and the fact that $0<\Pi_0(B_\delta^y)<1$, one concludes that
\begin{align*}
\Pi_n (\Theta\setminus B_\delta^y \mid y)
	& \,=\, \frac{ \int_{\Theta\setminus B_\delta^y}  \int_\Omega e^{\varphi_n(\theta,x,y)} \, d\mu_\theta(x) \, d \Pi_0 (\theta)}{ \int_\Theta  \int_\Omega e^{\varphi_n(\theta,x,y)} \, d\mu_\theta(x) \, d \Pi_0 (\theta)} \\
	& \leqslant \frac{ \int_{\Theta\setminus B_\delta^y}  \int_\Omega e^{\varphi_n(\theta,x,y)} \, d\mu_\theta(x) \, d \Pi_0 (\theta)}{ \int_{B^y_\delta}  \int_\Omega e^{\varphi_n(\theta,x,y)} \, d\mu_\theta(x) \, d \Pi_0 (\theta)} \\
%	& \leqslant \frac1{\Pi_0(B^y_\delta)} e^{-(\alpha^y-d_\delta)n} e^{(\alpha^y-\zeta)n} \\
	& \leqslant \frac1{\Pi_0(B^y_\delta)} e^{-(\zeta-d_\delta)n}
\end{align*}
tends exponentially fast to zero, as claimed.
Hence, any accumulation point of $(\Pi_n(\cdot\mid y))_{n\geqslant 1}$ (in the weak$^*$ topology) is supported on
the compact set $\text{argmax} \, \Gamma^y$, which proves the first statement in the theorem. As the second assertion is immediate from the first one, this concludes the proof of the theorem.
\hfill $\square$

%%%%%%%%%%%%%%%%%%%%%%%%
\subsection{Proof of Theorem~\ref{thm:main3}}

Consider the family of loss functions $\ell_n :\Theta \times X \times Y \to \mathbb R$ defined  by ~\eqref{eq:lossphi}
associated to an almost additive sequence $\Phi=(\varphi_n)_{n\geqslant 1}$
of continuous and non-negative observables $\varphi_n :\Theta \times X \times Y \to \mathbb R_+$
satisfying assumptions (H1)-(H2).
\begin{enumerate}
\item[(H1)] for each $\theta\in \Theta$ and $x\in X$ there exists a constant $K_{\theta,x}>0$ so that, for every $y\in Y$,
\begin{equation*}
\varphi_n(\theta,x,y) + \varphi_m(\theta,x,T^n(y)) - K_{\theta,x}
	\leqslant \varphi_{m+n}(\theta,x,y)
	\leqslant \varphi_n(\theta,x,y) + \varphi_m(\theta,x,T^n(y)) + K_{\theta,x}
\end{equation*}
\item[(H2)] $\int K_{\theta,x} d\mu_{\theta}(x)<\infty$ for every $\theta\in \Theta$.
\end{enumerate}
The \emph{a posteriori}  measures are
\begin{equation}\label{def:intpin}
\Pi_n (E \mid y)\,=\, \frac{ \int_E  \psi_n(\theta,y)\, d \Pi_0 (\theta)}{ \int_\Theta  \psi_n(\theta,y) \, d \Pi_0 (\theta)},
\end{equation}
where the sequence $\psi_n(\theta,y)=\int_\Omega \varphi_n(\theta,x,y) \, d\mu_\theta(x)$ is almost additive in the $y$-variable. Indeed, this family satisfies
%. This in the sense that the
%$\sigma$-invariance of $\mu_\theta$ and the almost-additivity condition
%$
%\varphi_{n}(x) + \varphi_{m}(\sigma^n(x)) - C
%	\leqslant \varphi_{m+n}(x)
%	\leqslant \varphi_{n}(x) + \varphi_{m}(\sigma^n(x)) +C
%$
%ensures that
$$
\psi_{n}(\theta,y) + \psi_{m}(\theta, T^n(y)) - \int K_{\theta,x} d\mu_{\theta}(x)
	\leqslant \psi_{m+n}(\theta,y)
	\leqslant  \psi_{n}(\theta,y) + \psi_{m}(\theta, T^n(y)) + \int K_{\theta,x} d\mu_{\theta}(x)
$$
for every $m,n\geqslant 1$, every $\theta\in \Omega$ and $y\in Y$.
Now, for each fixed $\theta\in \Theta$, we note that the sequence of observables
$$\Big(\psi_n(\theta,\cdot) + \int K_{\theta,x} d\mu_{\theta}(x)\Big)_{n\geqslant 1}$$
is subadditive. Hence, Kingman's subadditive ergodic theorem ensures that
the limit
$\lim_{n\to\infty} \frac{\psi_n(\theta, y)}n$
does exist and is $\nu$-almost everywhere constant to the non-negative function
$\psi_*(\theta):= \inf_{n\geqslant 1} \frac1n \int {\psi_n(\theta, y)}\, d\nu(y)$. The function $\psi_*$ is measurable and integrable, because it satisfies $0\leqslant \psi_*\leqslant \psi_1$.
Thus, taking the limit under the sign of the integral and noticing that the denominator is a normalizing term we conclude that
\begin{equation}\label{def:intpi}
\lim_{n\to\infty}\Pi_n (E \mid y) \,=\, \frac{ \int_E \psi_*(\theta)\, d \Pi_0 (\theta)}{ \int_\Theta \psi_*(\theta)\, d \Pi_0 (\theta)}
	\,=\, \frac{ \int 1_E \psi_*(\theta)\, d \Pi_0 (\theta)}{ \int_\Theta \psi_*(\theta)\, d \Pi_0 (\theta)}
\end{equation}
for every measurable subset $E\subset \Theta$. This proves the first statement of the theorem.

\medskip
We proceed to prove the level-1 large deviations estimates on the convergence of the \emph{a posteriori} measures $\Pi_n (\cdot \mid y)$ to $\Pi_*$, whenever $T$ is a subshift of finite type and $\nu$ is a Gibbs measure associated to a Lipschitz continuous potential $\varphi$. We will make use of the following instrumental lemma, whose proof is left as a simple exercise to the reader.

\begin{lemma}\label{leaux}
Given arbitrary functions $A,B: \Omega \to\mathbb R_+$ and constants $a,b,\delta>0$
and $0<\xi<b$, the following holds:
\begin{align*}
\Big\{\Big| \frac{A(y)}{B(y)} - \frac{a}{b}\Big|>\delta \Big\}
	\subset  S_1  \;\cup\; S_2  \;\cup\; S_3
\end{align*}
where
$
S_1= \Big\{\Big| {B(y)} - {b}\Big|>\xi \Big\},
\quad
S_2= \Big\{\frac1{b-\xi}\Big| {A(y)} - {a}\Big|>\frac\delta2 \Big\}
$
and
$
S_3= \Big\{\frac{a}{b(b-\xi)}\Big| {B(y)} - {b}\Big|>\frac\delta2 \Big\}.
$
\end{lemma}

\medskip
Let us return to the proof of the large deviation estimates. Given $g\in C(\Theta ,\mathbb R)$ it is not hard to check using \eqref{def:intpin} and
~\eqref{def:intpi} that
\begin{equation}\label{eq.ld1}
\int g\, d\Pi_n(\cdot\mid y)= \frac{ \int g(\theta) \frac{\psi_n(\theta,y)}n\, d \Pi_0 (\theta)}{ \int_\Theta  \frac{\psi_n(\theta,y)}n \, d \Pi_0 (\theta)}
	\quad\text{and}\quad
	\int g\, d\Pi_*= \frac{ \int g(\theta) \psi_*(\theta)\, d \Pi_0 (\theta)}{ \int_\Theta \psi_*(\theta)\, d \Pi_0 (\theta)}.
\end{equation}
Fix $\delta>0$. In order to provide an upper bound for
$$
\limsup_{n \to \infty} \frac{1}{n}\,\,  \log \,\nu ( \{\,y \in \Omega :
	\Big|\int g\, d\Pi_n (\cdot \mid y) - \int g\, d\Pi_*\Big|>\delta  \})
$$
we will estimate the set $\Big\{\Big|\int g\, d\Pi_n (\cdot \mid y) - \int g\, d\Pi_*\Big|>\delta\Big\}$ as in Lemma~\ref{leaux}.
For that purpose, fix $0<\xi <\min_{\theta\in\Theta} \psi_*(\theta)$.
%\marginpar{\tiny\color{red} hypothesis}
%
For each fixed $\theta\in \Theta$ the family $\Psi^\theta:=(\psi_n(\theta,\cdot))_n$ is almost-additive. Hence Theorem~\ref{thm.deviations} implies that
$$
\limsup_{n \to \infty} \frac{1}{n}\,\,  \log \,\nu ( \{\,y \in \Omega :
	\Big| \frac{\psi_n(\theta,y)}n - \psi_*(\theta)\Big|\geqslant \xi  \})
	\leqslant \sup_{\mathcal P^1_{\theta,\xi,\delta}} \Big\{-P(\sigma,\varphi)+h_\eta(\sigma) + \int \varphi\, d\eta \Big\}
$$
where $\mathcal P^1_{\theta,\xi,\delta}\subset {\mathcal M}_\sigma(\Omega)$ is the space of invariant probability measures $\eta$ such that
$|{\mathcal F}(\eta,\Psi^\theta) -\psi_*(\theta)|\geqslant \xi$.
In consequence,
\begin{align}
\limsup_{n \to \infty} & \frac{1}{n} \log \,\nu \Big( \Big\{\,y \in \Omega :
	\Big| \int_\Theta  \frac{\psi_n(\theta,y)}n \, d \Pi_0 (\theta) - \int_\Theta \psi_*(\theta)\, d \Pi_0 (\theta)\Big|\geqslant\xi  \Big\}\Big)
	\nonumber \\
		& \leqslant \limsup_{n \to \infty}  \frac{1}{n}\,\,  \log \,\nu \Big( \Big\{\,y \in \Omega :
	\int_\Theta  \Big|  \frac{\psi_n(\theta,y)}n  - \psi_*(\theta)\,\Big|\,  d \Pi_0 (\theta)\geqslant\xi  \Big\}\Big)  \nonumber  \\
			& \leqslant \limsup_{n \to \infty}  \frac{1}{n}\,\,  \log \,\nu \Big( \Big\{\,y \in \Omega :
	  \Big|  \frac{\psi_n(\theta,y)}n  - \psi_*(\theta)\,\Big|\, \geqslant\xi, \; \text{for some}\; \theta\in \Theta \Big\}\Big)  \nonumber \\
			& \leqslant
	\sup_{\theta\in \Theta} \sup_{\mathcal P^1_{\theta,\xi,\delta}} \Big\{-P(\sigma,\varphi)+h_\eta(\sigma) + \int \varphi\, d\eta \Big\}.
	\label{conclusao1}
\end{align}
Analogously,
\begin{align}
\limsup_{n \to \infty} & \frac{1}{n}  \log \,\nu \Big( \Big\{\,y \in \Omega :
	\frac1{\int_\Theta \psi_*(\theta)\, d \Pi_0 (\theta)-\xi} \Big| \int g(\theta) \frac{\psi_n(\theta,y)}n\, d \Pi_0 (\theta) - \int_\Theta g(\theta) \psi_*(\theta)\, d \Pi_0 (\theta)\Big|\geqslant\frac\delta2  \Big\}\Big)
	\nonumber \\
& \leqslant \limsup_{n \to \infty}  \frac{1}{n}  \log \,\nu \Big( \Big\{\,y \in \Omega :
	 \Big| \int \frac{\psi_n(\theta,y)}n\, d \Pi_0 (\theta) - \int_\Theta \psi_*(\theta)\, d \Pi_0 (\theta)\Big|
	 	\geqslant\frac{\int_\Theta \psi_*(\theta)\, d \Pi_0 (\theta)-\xi}{2\|g\|_\infty} \delta   \Big\}\Big)
	\nonumber \\
			& \leqslant
	\sup_{\theta\in \Theta} \sup_{\mathcal P^2_{\theta,\xi,\delta}} \Big\{-P(\sigma,\varphi)+h_\eta(\sigma) + \int \varphi\, d\eta \Big\},
	\label{conclusao2}
\end{align}
where $\eta\in \mathcal P^2_{\theta,\xi,\delta} \subset {\mathcal M}_\sigma(\Omega)$ if and only if
$|{\mathcal F}(\eta,\Psi^\theta) -\psi_*(\theta)|\geqslant \frac{\int_\Theta \psi_*(\theta)\, d \Pi_0 (\theta)-\xi}{2\|g\|_\infty} \delta$.
The third term in the decomposition of Lemma~\ref{leaux} is identical to the estimate of ~\eqref{conclusao1} and we have
\begin{align}
\limsup_{n \to \infty} & \frac{1}{n} \log \,\nu \Big( \Big\{\,y \in \Omega :
	\Big| \int_\Theta  \frac{\psi_n(\theta,y)}n \, d \Pi_0 (\theta) - \int_\Theta \psi_*(\theta)\, d \Pi_0 (\theta)\Big|
		\geqslant\frac{(\int_\Theta \psi_*(\theta)\, d \Pi_0 (\theta)-\xi)^2}{ 2\int_\Theta \psi_*(\theta)\, d \Pi_0 (\theta)} \delta  \Big\}\Big)
	\nonumber \\
			& \leqslant
	\sup_{\theta\in \Theta} \sup_{\mathcal P^3_{\theta,\xi,\delta}}
		\Big\{-P(\sigma,\varphi)+h_\eta(\sigma) + \int \varphi\, d\eta \Big\},
	\label{conclusao3}
\end{align}
where $\eta\in \mathcal P^3_{\theta,\xi,\delta} \subset {\mathcal M}_\sigma(\Omega)$ if and only if
$|{\mathcal F}(\eta,\Psi^\theta) -\psi_*(\theta)|\geqslant \frac{(\int_\Theta \psi_*(\theta)\, d \Pi_0 (\theta)-\xi)^2}{ 2\int_\Theta \psi_*(\theta)\, d \Pi_0 (\theta)} \delta$.
Altogether, if $0<\delta<1$ and $\xi=\delta \cdot \min\{ \inf_{\theta\in\Theta}\psi_*(\theta),\int_\Theta \psi_*(\theta)\, d \Pi_0 (\theta)\} >0$, estimates \eqref{conclusao1}-\eqref{conclusao3} imply that there exists $c>0$ so that
%for each $0<\xi<\inf_{\theta\in\Theta}\psi_*(\theta)$
%NEARLY BUT NOT EXACTLY
%$c=\min\{ \frac{(\int_\Theta \psi_*(\theta)\, d \Pi_0 (\theta))^2}{ 8\int_\Theta \psi_*(\theta)\, d \Pi_0 (\theta)},
%\frac{\int_\Theta \psi_*(\theta)\, d \Pi_0 (\theta)}{4\|g\|_\infty}\}$
%taking $\xi$ half the value
%$c=\min\{ \frac{\int_\Theta \psi_*(\theta)\, d \Pi_0 (\theta)}{ 8},
%\frac{\int_\Theta \psi_*(\theta)\, d \Pi_0 (\theta)}{4\|g\|_\infty}\}$.

\begin{align*}
\limsup_{n \to \infty} \frac{1}{n}\,\,  & \log \, \nu \Big( \Big\{\,y \in \Omega :
	\Big|\int g\, d\Pi_n (\cdot \mid y) - \int g\, d\Pi_*\Big|\geqslant\delta  \Big\}\Big) \\
			& \leqslant
		\sup_{\theta\in \Theta}
		\max_{1\leqslant i \leqslant 3}
		\sup_{\mathcal P^i_{\theta,\xi,\delta}}
			\Big\{-P(\sigma,\varphi)+h_\eta(\sigma) + \int \varphi\, d\eta \Big\} \\
			& \le  \sup_{\theta\in \Theta}  \;
			 \sup_{\{\eta\colon |{\mathcal F}(\eta,\Psi^\theta) -\psi_*(\theta)|\geqslant c\delta\}} 					
			 	\Big\{-P(\sigma,\varphi)+h_\eta(\sigma) + \int \varphi\, d\eta \Big\}		
%		& \leqslant
%		\sup_{\theta\in \Theta}
%		\sup_{\xi>0}
%		\max_{1\leqslant i \leqslant 3}
%		\sup_{\mathcal P^i_{\theta,\xi,\delta}}
%			\Big\{-P(\sigma,\varphi)+h_\eta(\sigma) + \int \varphi\, d\eta \Big\} \\
%				& \le  \sup_{\theta\in \Theta}
%			\max
%			\Big\{
%			\sup_{\{\eta\colon |{\mathcal F}(\eta,\Psi^\theta) -\psi_*(\theta)|\geqslant \xi\}}
%				\big\{-P(\sigma,\varphi)+h_\eta(\sigma) + \int \varphi\, d\eta \big\}, \;
%				\\
%			& \hspace{2cm} \sup_{\{\eta\colon |{\mathcal F}(\eta,\Psi^\theta) -\psi_*(\theta)|\geqslant c\delta\}} 						\big\{-P(\sigma,\varphi)+h_\eta(\sigma) + \int \varphi\, d\eta \big\}		
%			\Big\}
\end{align*}
Finally, it remains  to guarantee that the right hand-side above is strictly negative. Notice that as
${\mathcal F}(\nu,\Psi^\theta)=\psi_*(\theta)$, the uniqueness of the equilibrium state (which is an invariant Gibbs measure) for the potential $\varphi$ and the continuity of the map $\eta\mapsto {\mathcal F}(\eta,\Psi^\theta)$ imply that
the set
$\mathcal B_\theta(\delta):=\{\eta\in \mathcal M_\sigma(\Omega)\colon |{\mathcal F}(\eta,\Psi^\theta) -\psi_*(\theta)|  \geqslant c\delta\}$
is compact and disjoint from $\{\nu\}$,
hence
%$$
%			\beta(\theta):=
%			 \sup_{\{\eta\colon |{\mathcal F}(\eta,\Psi^\theta) -\psi_*(\theta)|\geqslant c\delta\}} 					
%			 	\Big\{-P(\sigma,\varphi)+h_\eta(\sigma) + \int \varphi\, d\eta \Big\}	
%			 <0.	
%$$
$
d_{\mathcal M_\sigma(\Omega)}\big(\nu, \mathcal B_\theta(\delta)\big)
	> 0, \text{for each $\theta\in \Theta$}.
$
Hence, under the additional assumption that both maps $\theta\mapsto {\mathcal F}(\eta,\Psi^\theta)=\inf_{n\ge 1}\frac1n \int \psi_n(\theta,\cdot)d\eta$ and $\theta\mapsto \psi_*(\theta)={\mathcal F}(\nu,\Psi^\theta)$ are
continuous
we conclude that
$$
\min_{\theta\in \Theta} d_{\mathcal M_\sigma(\Omega)}\big(\nu, \mathcal B_\theta(\delta)\big)
	> 0
$$
and, consequently,
$$
\sup_{\theta\in \Theta}  \;
			 \sup_{\{\eta\colon |{\mathcal F}(\eta,\Psi^\theta) -\psi_*(\theta)|\geqslant c\delta\}} 					
			 	\Big\{-P(\sigma,\varphi)+h_\eta(\sigma) + \int \varphi\, d\eta \Big\}<0,
$$	
which finishes the proof of the theorem
\hfill $\square$

%\end{document}

\subsection*{Acknowledgments}
The authors are indebted to the anonymous referees for the careful rea\-ding of the manuscript and many suggestions that helped to improve
the presentation of the paper.
AOL and SRCL was partially supported by CNPq grant.
PV was partially supported by CMUP (UID/MAT/00144/2019), which is funded by FCT with national (MCTES) and European structural funds through the programs FEDER, under the partnership agreement PT2020, and
by Funda\c c\~ao para a Ci\^encia e Tecnologia (FCT) - Portugal through the grant CEECIND/03721/2017 of the Stimulus of
Scientific Employment, Individual Support 2017 Call.

\end{document}